\documentclass[a4paper,oneside,11pt]{article} % Indique que le document est de type standard ("article"), au format A4 ("a4"), en recto simple ("oneside"), avec des fontes de taille moyenne ("11pt").
\usepackage[utf8]{inputenc} % Encodage des caractères du fichier-source.
\usepackage[T1]{fontenc} % Encodage des caractères en sortie.
\usepackage{textcomp} % Jeu de symboles complémentaires.
\usepackage[english]{babel} % Support de la langue française.
\usepackage{indentfirst}
\usepackage[autolanguage]{numprint} %% Formatage des nombres.
\usepackage{hyperref} % Génère des liens hypertexte dans le fichier pdf.
\hypersetup{colorlinks=true,linkcolor=blue,citecolor=red,urlcolor=blue}
\usepackage{graphicx} % Permet l'insertion d'images.
\usepackage{verbatim} % Définit des environnements de texte préformaté et de commentaires.
\usepackage{makeidx} % Permet la création d'un index
\usepackage[refpage]{nomencl} % Permet la création d'un index des notations
\usepackage{enumitem} % Permet les listes
\usepackage{tikz} % Pour les dessins en Tikz
\usetikzlibrary{matrix,arrows,patterns}
\usepackage{fullpage}
\usepackage{comment}
\usepackage{stmaryrd}
\usepackage{ulem}
\usepackage{tikz}
\usepackage[mathscr]{euscript}

\usepackage{amsmath,amssymb,amsthm,amsfonts} % Toutes ces extensions proviennent de la classe AMS-LaTeX.
\usepackage[all]{xy} % pour des diagrammes (-commutatifs)

\numberwithin{equation}{section}

\newcounter{Main}

\theoremstyle{plain} % Pour tous ce qui ressemble à des théorèmes.

\newtheorem{MainThm}[Main]{Theorem} % Théorèmes principaux
\swapnumbers % Avec cette ligne, le numéro des théorèmes s'affiche avant le mot "Théorème".
\theoremstyle{definition} % Pour tout ce qui ressemble à des définitions : définitions, notations...
\newtheorem{Def}{Definition}[section] % Définitions. Avec l'option "[section]", les définitions (et tout le reste) seront numérotées en fonction de la section courante.
\newtheorem{Def,Thm}[Def]{Definition and theorem} %Pour ce qui est une définition par un théorème
\newtheorem{Def,Prop}[Def]{Proposition-definition} %Pour ce qui est une définition par un théorème
\newtheorem{Not}[Def]{Notation}
\theoremstyle{plain} % Pour tous ce qui ressemble à des théorèmes.
\newtheorem{prop}[Def]{Proposition} % Propositions. Avec l'option "[Def]", les propositions (et tout le reste) seront numérotées collectivement avec les définitions.
\newtheorem{lem}[Def]{Lemma} % Lemmes.

\newtheorem{theo}[Def]{Theorem} % Théorèmes.
\newtheorem{Corollary}{Corollary}[Def] % Corollaires.
\theoremstyle{remark} % Pour tout ce qui ressemble à des remarques.
\newtheorem{ex}[Def]{Example} % Exemples.
\newtheorem{Remark}[Def]{Remark} % Remarques.
\newcommand{\LL}{\mathscr{L}}

\newcommand{\ZZ}{\mathbb{Z}}

\newcommand{\DD}{\mathfrak{D}}

\newcommand{\oo}{\mathcal{O}}

\newcommand{\ap}{\Sigma}

\newcommand{\BB}{\vec{B}}
\newcommand{\BBB}{(\!(\vec{B})\!)}

\newcommand{\A}{R}

\newcommand{\R}{\mathbb{R}}

\usepackage{xcolor}
\usepackage[all]{xy}

\newcommand\bmattrix[4]{\left(\begin{array}{cc}#1&#2\\#3&#4\end{array}\right)}

\usepackage{mathtools}

\title{A quotient of the Bruhat–Tits building classifying maximal matrix orders over curves}

\author{Luis Arenas-Carmona\footnote{Universidad de Chile, Facultad de Ciencias, Casilla 652, Santiago, Chile. Email: learenas@u.uchile.cl.} \,, Claudio Bravo\footnote{Universidad de Talca, Instituto de Matemáticas, Talca, Chile. Email: claudio.bravo@utalca.cl.} 
\, and Ignacio Saavedra\footnote{Facultad de Ingeniería y Ciencias, Universidad Diego Portales, Av. Ejercito 441, Santiago, Chile. Email: ignacio.saavedra@udp.cl.}} % Remplacez "Auteur" par le nom du ou des auteurs.
\date{}

\begin{document} % Début du document proprement dit.

\maketitle

\begin{abstract}

In a previous work, we defined a quotient graph of the 
Bruhat-Tits tree associated to the degree two special linear 
group over a local field whose vertex set is in correspondence 
with the maximal orders in a certain number of spinor genera.
This graph has been applied to the study of the
structure of some arithmetic groups, their presentations and their
(co)homology groups. The genera of Eichler orders,
which, for instance, are related to the problem of finding a basis
that simultaneously splits two vector bundles with the same 
generic fiber, can often be described in terms of simpler genera
via quotients. Furthermore, some embedding problems of
orders can be studied via quotient graphs. These results
provide a geometric framework for the study of maximal orders over function fields.

In this article, we 
extend that construction to a higher dimensional setting through
the use of singular chain complexes. Then we describe the set of
vertices in this quotient that correspond to split maximal orders
and the topology of the corresponding region of the quotient 
complex. More precisely, we show that these ``split vertices''
fall into a finite number of quotients of apartments by suitable
subgroups of the corresponding Weyl group, the symmetric group. 
We relate these ``quotient apartments'' with the cuspidal sectors 
defined by previous authors who studied the quotient by the group
of integral points of the special linear group. Both quotients
are indeed related by a ramified cover that is unramified
on cusps, as proved here.
\\



\end{abstract}

\section{Introduction}\label{section intro}

The study of arithmetic subgroups of reductive algebraic groups is a central 
theme in modern mathematics, with connections to automorphic forms,
representations, quadratic forms, (co)homology, $K$-theory, complex geometry, 
combinatorial group theory, class field theory, orders (arithmetic of algebras)
and dynamical systems, among other domains. A fundamental approach to 
understanding these subgroups consists in analyzing their action on the associated 
symmetric spaces or, in the non-archimedean setting, on Bruhat-Tits buildings. 
This approach, known as reduction theory \cite{BorelArithmeticsubgroup}, 
seeks to describe the fundamental domains of such actions and, more generally, the 
structure of the corresponding orbit spaces.

In order to study reductive (algebraic) groups over local or global fields, Tits 
introduced certain complexes, called buildings, that are analogous to the symmetric 
spaces associated to reductive groups over $\mathbb{R}$.
More precisely, to any Henselian discretely valued field $\hat{K}$ and any reductive 
$\hat{K}$-group $\mathtt{G}_{\hat{K}}$, Bruhat and Tits associated a polysimplicial 
complex $\mathcal{B} = \mathcal{B}(\mathtt{G}_{\hat{K}})$, called the Bruhat-Tits 
building of $\mathtt{G}_{\hat{K}}$ (cf.~\cite{BT1,BT2}).
This complex can be constructed by gluing an infinite number of copies of the affine 
space $\mathbb{R}^{r}$, where $r$ is the (split) rank of $\mathtt{G}_{\hat{K}}$, 
which is endowed with a semi-regular tessellation.
In the rank one context, i.e., for matrix groups such as $\mathrm{SL}_2$ 
or $\mathrm{SU}_3$, the associated 
Bruhat-Tits building is a semi-homogeneous tree.

In this work we
focus on the function field setting.
Indeed, let $X$ be a smooth projective curve over a field $k$, let 
$P$ be a closed point of $X$, and let $K_P$ be the completion of $K=k(X)$ with respect to the
discrete valuation $\nu_P:K \to \mathbb{Z} \cup \lbrace \infty \rbrace$ induced by $P$.
Let us also denote by $\A=\mathcal{O}_X(X \smallsetminus \{P\})$
the ring of functions on $X$ that are regular outside $\lbrace P \rbrace$.
In \cite{SerreTrees} Serre gives a complete description of the action of 
$\mathrm{SL}_2(\A)$ on its Bruhat-Tits tree. His result can be phrased as follows:

\begin{MainThm}\cite[Ch. II, Th. 9]{SerreTrees}\label{teo serre quot}
Let $\mathcal{T}=\mathcal{B}(\mathrm{SL}_{2,K_P})$ be the Bruhat-Tits tree of 
$\mathrm{SL}_{2,K_P}$.
The (quotient) graph $\mathrm{SL}_2(\A) \backslash \mathcal{T}$ is the union of a 
connected graph $Y$ of finite diameter, and a family of rays called cusps, which are 
indexed by the Picard group $\mathrm{Pic}(\A)\cong
\mathrm{Pic}(X)/\langle[P]_{X}\rangle$, where $[P]_{X}$ is the class in 
$\mathrm{Pic}(X)$ of the divisor $P$. Moreover, $Y$ is finite whenever $k$ is.
\end{MainThm}

This result has several important consequences on arithmetic subgroups of $\mathrm{SL}_2$. 
For example, by applying Bass-Serre theory  (cf.~\cite[Ch. I, \S 5]{SerreTrees}) on the quotient graphs described in Th.~\ref{teo serre quot}, Serre shows that $\mathrm{SL}_2(\A)$ can be expressed as an amalgamated product of simpler subgroups (cf.~\cite[Ch. II, \S 2, Th. 10]{SerreTrees}).
This structural description has significant consequences for the (co)homology of 
$\mathrm{SL}_2(\A)$ and its finite-index subgroups (cf.~\cite[Ch. II, \S 2.8-2.9]{SerreTrees}).

%Using this results, Serre describes in ?? presentations of  $\mathrm{SL}_2(\A)$, with significant consequences for the (co)homology of $\mathrm{SL}_2(\A)$ and its finite-index subgroups.
In our previous work, the graph described in 
Th. \ref{teo serre quot} is called 
the S-graph of $\mathbb{M}_2(K)$ at $P$. We have focused 
on a slightly different graph, called the classification graph or C-graph,
whose vertex set corresponds
naturally to some set $\Omega$ of conjugacy classes of 
maximal $X$-orders, i.e., sheaves of orders in $\mathbb{M}_2(K)$ 
that are maximal at any affine open set $U \subset X$. To be precise, 
$\Omega$ is the set of conjugacy classes in one or two spinor genera 
of maximal orders, depending on whether a suitable Frobenius map at $P$ 
is trivial or not. This C-graph is defined as a quotient 
$N \backslash \mathcal{T}$ of the Bruhat-Tits tree by the normalizer of
a suitable affine order, and naturally is covered by the S-graph, which 
is defined as the quotient by the normal subgroup $\Gamma=K^*\mathrm{GL}_2(R)$. 
Furthermore, this construction can be applied to 
each spinor genus or pair of such to obtain a potentially disconnected graph 
that encloses information on all maximal $X$-orders, which we call the full 
C-graph, see \cite{ARENASCARMONA2014258} for details. The C-graph approach offers 
a distinctly geometric framework that has proven valuable in studying orders over 
function fields. Specifically, 
it has been used to analyze the structure of arithmetic groups, 
predominantly Hecke congruence subgroups, as in \cite{BravoHecke}, and 
to describe genera of Eichler orders in terms of simpler genera 
\cite{BravoArenasQH}. The genera of Eichler orders directly relate to
the task of finding a basis that simultaneously splits two vector 
bundles sharing a generic fiber. Furthermore, quotient graphs have 
successfully advanced the study of embedding problems for orders. The 
latter has been mostly done in the number field setting 
\cite{Arenascubic}, but the theory underlying these applications works 
for all global fields. We have a structural result for this quotient in 
terms of split maximal orders, i.e., orders of the form 
$$\DD_B=\bmattrix{\oo_X}{\LL^B}{\LL^{-B}}{\oo_X},$$
where $\LL^{B}(U)=\lbrace f \in K \big{\vert} (\mathrm{div}(f)+B) 
\vert_U \geq 0 \rbrace$, namely:

\begin{MainThm}\cite[Th. 1.2]{ARENASCARMONA2014258}\label{MThB}
In the full C-graph of $\mathbb{M}_2(K)$ at $P$, the vertices 
corresponding  to split maximal orders are located in a finite 
disjoint union of infinite lines or half-lines. The set of such 
lines is in correspondence with the pairs of the form $\lbrace 
a, -a \rbrace$ in $\mathrm{Pic}(\A)$. A split order 
$\DD_{B}$ is in the line corresponding to $a=\bar{B} + 
\langle \bar{P} \rangle$. The half lines correspond to the 
elements $a$ of order $1$ or $2$.
\end{MainThm}

Note that almost all the previous results are specific to arithmetic subgroups of 
$\mathrm{SL}_2$, $\mathrm{GL}_2$ or even $\mathrm{PGL}_2$.
Thus, it is natural to ask whether or not, or how, the preceding result extends 
to higher dimensional algebraic groups, such as $\mathrm{SL}_n$, $\mathrm{GL}_n$ 
or even algebraic groups such as $\mathrm{SO}_n$, $\mathrm{Sp}_{2n}$ or 
$\mathrm{SU}_n$. In the sequel, given a variety $\mathtt{H}$ defined over 
a ring $R'$, we denote by $\mathtt{H}_{R''}=\mathtt{H} \otimes_{R'} R''$ its 
scalar extension defined by a ring homomorphism $R'\to R''$.
We say that $\mathtt{G}_K$ is a $K$-split group when it contains a maximal 
$K$-torus that is isomorphic to $\mathbb{G}_m^{r}$, 
where $\mathbb{G}_m$ denotes the multiplicative group $\mathrm{GL}_1$.
Examples of such groups $\mathtt{G}_K$ are $\mathrm{GL}_n$, 
$\mathrm{SL}_{n}$, $\mathrm{PGL}_n$, $\mathrm{Sp}_{2n}$, 
$\mathrm{SO}_n$, $\mathrm{Spin}_n$, among others.
Any reductive $K$-split group $\mathtt{G}_K$ has a $\mathbb{Z}$-model 
$\mathtt{G}=\mathtt{G}_{\mathbb{Z}}$ \cite[Exp.~XXV~1.3]{SGA3-3}, 
i.e., a $\mathbb{Z}$-group scheme such that 
$\mathtt{G}\otimes_{\mathbb{Z}} K=\mathtt{G}_K$. Such a group (scheme) 
is called the Chevalley model of $\mathtt{G}_K$.
In particular, $\mathtt{G}(\A)$ can be defined as the group of 
$\A$-points of $\mathtt{G}_{\mathbb{Z}}$.

An $(n-1)$-dimensional building is a cellular complex defined by gluing multiple 
copies of a complex $\mathcal{A}$, called an apartment, whose underlying 
topological space $\mathbf{A}$, or realization, is homeomorphic to affine space, 
i.e., it is canonically homeomorphic to $\mathbb{R}^{n-1}$ once a base-point 
$o\in\mathbf{A}$, called the origin, is chosen. This endows $\mathbf{A}$
with a vector space structure. The analogous role to that 
of a ray in a tree is played by sector chambers, as defined in \cite[\S 1.4]{AbramenkoBrown}. 
See Fig. \ref{f2}. A sector chamber $\mathbf{Q}\subseteq\mathbf{A}$ 
is defined here as a shift $\mathbf{Q}=x+\mathbf{S}$ of a suitable open fundamental 
region $\mathbf{S}$ for the natural action of the permutation group $S_n$ on 
$\mathbf{A}$ fixing the base point $o$. The point $x$ is called the tip of the 
sector. When $x$ is a vertex,  the closure $\overline{\mathbf{Q}}$ of $\mathbf{Q}$ 
in $\mathbf{A}$ is indeed the realization of a complex called a closed sector 
chamber. Note that \cite{AbramenkoBrown} assumes that $o$ itself is a vertex,
but we prefer to keep it general since it is needed later.
When $X=\mathbb{P}^1_{k}$ and $\deg(P)=1$, i.e., 
$\A \cong k[t]$, and $\mathtt{G}$ is a Chevalley model of a 
semisimple (s.s.) simply connected (s.c.) $K$-split group, 
Soul\'e describes the corresponding quotient as follows:

\begin{MainThm}
\cite[\S 1, Th. 1]{So}\label{teo soule quot}
In the preceding hypotheses, the (quotient) complex $\mathtt{G}(\A)
\backslash \mathcal{B}(\mathtt{G}_{K_P})$ is isomorphic to a closed sector chamber of $\mathcal{B}(\mathtt{G}_{K_P})$.
\end{MainThm}

In the same work, Soul\'e also describes $\mathtt{G}(\A)$ as an amalgamated product of well-understood subgroups (cf.~\cite[\S 2, Th. 3]{So}) and studies its homology groups over suitable modules of coefficients.
Th. \ref{teo soule quot} has been extended by Margaux in \cite{Margaux} to the context of isotrivial s.c. s.s. $k$-groups, which are s.c. s.s. groups defined over $k$ and therefore splitting over some finite extensions of the ground field $k$.
Soul\'e's theorem has also important applications to $K$-theory. In particular, Knudson and Wendt 
\cite{Knudson1, Knudson2, KnudsonBook, Wendt} used related techniques to address 
the homotopy invariance question for the homology of Chevalley groups over polynomial rings.

Let $S$ be a finite set of closed points of $X$ and let $\tilde{\mathtt{G}}_K$ be an arbitrary linear $K$-group, meaning that
there exists a $K$-representation $\rho: \tilde{\mathtt{G}}_K \to 
\mathrm{SL}_{n}$, for some $n>0$.
Then, we can define $\tilde{\mathtt{G}}_K(\mathcal{O}_{S})$ as 
$\tilde{\mathtt{G}}_K(K) \cap \rho^{-1}
\big( \mathrm{SL}_n(\mathcal{O}_{S}) \big)$.
The group so defined depends on the representation $\rho$.
For instance, for $\tilde{\mathtt{G}}_K=\mathrm{SL}_n$, we assume that 
$\rho$ is the identity, while, for 
$\tilde{\mathtt{G}}_K=\mathrm{PSL}_n$, we make it act on 
the matrix algebra
$\mathbb{M}_{n,K}$ via conjugation.
An $S$-arithmetic group $G$ of $\tilde{\mathtt{G}}_K$ 
is a subgroup of $\tilde{\mathtt{G}}_K(K)$ that is commensurable 
with $\tilde{\mathtt{G}}_K(\mathcal{O}_{S})$.
In order to study these groups, we can use their diagonal action on 
the product of Bruhat-Tits buildings 
$\mathcal{B}_{S}:=\prod_{P \in S} \mathcal{B}
(\tilde{\mathtt{G}}_{K_P})$.
Quotients of the form $\mathrm{SL}_2(\mathcal{O}_{S}) 
\backslash \mathcal{B}_{S}$ have been described in 
\cite[Prop. 1]{Stuhler}.
Assuming that $k$ is a finite field and $\mathtt{G}$ is an 
arbitrary isotropic $K$-group, Bux, Köhl and Witzel describe some 
finiteness properties of the (quotient) space $G \backslash 
\mathcal{B}_{S}$ by exhibiting the following cover:

\begin{MainThm}\cite[Prop 13.6]{B}\label{teo cover witzel}
Assume that $k$ is finite. Let $\mathtt{G}$, $G$ and $\mathcal{B}_{S}$ 
be as above.
Then, there exists a constant $\kappa$ and finitely many sector 
chambers $\mathbf{Q}_1, \cdots, \mathbf{Q}_s$ of $\mathcal{B}_{S}$ 
such that:
\begin{itemize}
\item the $G$-translates of the $\kappa$-neighborhood of 
$\bigcup_{i=1}^{s} \mathbf{Q}_i$ cover $\mathcal{B}_{S}$, and
\item for $i\neq j$, the $G$-orbits of $\mathbf{Q}_i$ and 
$\mathbf{Q}_j$ are disjoint.
\end{itemize} 
\end{MainThm}
Above, the $\kappa$-neighborhood of a set $\mathbf{X}$ is the
union of all simplices $\mathbf{T}$ for which there is a sequence
$\mathbf{T}=\mathbf{T}_0,\dots,\mathbf{T}_n$ of simplices, where
consecutive simplices share a vertex, and $\mathbf{T}_n\cap 
\mathbf{X}\neq\emptyset$.
More technical coverings in this spirit were previously described by 
Harder in \cite{Harder1}.
Important evidence for Th. \ref{teo cover witzel} was also contributed by
Behr~\cite{Behr}, Abels~\cite{Abels}, Abramenko~\cite{Abramenko}, as well as
Bux and Wortman~\cite{Bux-Wortman}.
Some relevant finiteness properties of $S$-arithmetic subgroups, 
including the homological finiteness length, have been studied by using 
the previously cited results on actions on buildings 
\cite{Abels,Abramenko,Behr,B,Bux-Wortman,Gandini,Stuhler}.
They have important consequences on the (co)homology of the involved 
groups according to~\cite[Prop. 2.15]{Bieri}, \cite{H2}, \cite{Wortman}, 
among other works.

More recently, Kondo and Yasuda \cite{KondoYasuda} obtained further 
descriptions 
of quotients of the form $\mathrm{GL}_n(\A) \backslash\mathcal{B}
(\mathrm{GL}_{n,K_P})$, motivated by the study of automorphic forms for 
$\mathrm{GL}_n(\A)$.
For the case of $\mathrm{PGL}_3$ over a rational function field, some 
explicit descriptions (mainly for $X=\mathbb{P}^1_k$ and $\deg(P)\in 
\lbrace2,3\rbrace$) of 
the quotient graphs $\mathrm{PGL}_3(\A)\backslash \mathcal{X}$ were obtained by 
Christ~\cite{Christ}, where $\mathcal{X}$ denotes the $1$-skeleton of the 
associated Bruhat-Tits building $\mathcal{B}(\mathrm{GL}_{3,K_P})$.
Further examples of quotients of Bruhat-Tits buildings (and trees) for 
arithmetic subgroups of $\mathrm{SL}_n(\A)$ (and $\mathtt{G}(\A)$, for more 
involved algebraic 
groups $\mathtt{G}$) can be found in 
\cite{Bau,Bockle,BravoHecke,BravoHI,Kohl,Mason2,M,
MasonPremetSuryZalesskii,MasonSchw1,MasonSchw2}.

Recall that the star of a simplex $\mathbf{T}$ in a simplicial complex 
$\mathcal{X}$ is the union, in the realization $\mathbf{X}$ of
$\mathcal{X}$, of the interior of every simplex that has $\mathbf{T}$
as a face. This is an open neighborhood of every interior point of
$\mathbf{T}$. In this language,
a precise topological (and combinatorial) description of the 
sector chambers in Th. \ref{teo cover witzel}, called cuspidal chambers,
is given by Theorem \ref{MthE} below. Roughly speaking, 
this result shows that sector chambers have a big piece that does not branch in the quotient.

\begin{MainThm}\cite[Ths.~2.1-2.2]{BravoLoisel}\label{MthE}
Let $\mathtt{G}$ be a Chevalley model of a $K$-split semisimple group
and let $\mathcal{B}=\mathcal{B}(\mathtt{G}_{K_P})$ be the associated 
building at $P$. Let $\mathbf{Q}$ be a sector chamber of $\mathcal{B}$.
Then there exists a subsector chamber $\mathbf{Q}' \subseteq \mathbf{Q}$
such that the star in $\mathcal{B}$ of every simplex in $\mathbf{Q}'$ is covered by the $\mathtt{G}(\A)$-orbit of its star in $\mathbf{Q}$.
Moreover, there exists a subsector chamber $\mathbf{Q}'' \subseteq
\mathbf{Q}$ which embeds into the quotient $\mathtt{G}(\A)\backslash 
\mathcal{B}$.
\end{MainThm}

The preceding result holds even when the ground field $k$ is infinite according to \cite[Th.~2.1]{BravoLoisel}.
A sector chamber like $Q''$, satisfying these 
``flatness'' and embedding conditions, is called cuspidal. 
This definition extends easily to other quotients. 
%See Def. \ref{Definition cusps} for more details.
A precise description of the quotient $\mathtt{G}(\A)\backslash \mathcal{B}(\mathtt{G}_{K_P})$ is given in \cite[Th.~2.4]{BravoLoisel}, 
for general $K$-split semisimple groups.
More 
precisely, the latter result shows that $\mathtt{G}(\A)\backslash \mathcal{B}(\mathtt{G}_{K_P})$ is obtained by 
gluing a connected CW-complex to a family of sector chambers, called cuspidal rational sector chambers, whose $\mathtt{G}(\A)$-orbits 
are pairwise disjoint and each of which embeds homeomorphically into the quotient. 
The set of such chambers is in bijection with $\mathrm{Pic}(\A)^{r}$, where $r$ is the rank of $\mathtt{G}$.
This result simultaneously generalizes Th. 
\ref{teo serre quot} and Th. \ref{teo soule quot} to 
arbitrary Chevalley groups over rings of the form $\A$, 
and refines Th. \ref{teo cover witzel} by providing an 
explicit description of a fundamental domain rather than 
a mere covering. Some progress has also been made beyond the
Chevalley setting, for quasi-split non-split algebraic groups of low rank 
such as $\mathrm{SU}_3$ in \cite{ArenasBravoLoiselLucchini}.

The structure of quotients of Bruhat-Tits buildings, as described
in particular in Th.~\ref{teo serre quot} and Th.~\ref{teo soule quot}, has 
important 
consequences for the study of modular and automorphic forms of Drinfeld type.
Applying these theorems, Gekeler and collaborators developed 
the latter theory, especially in higher rank, in works such as 
\cite{gek1,gek5,gek2,gek6,Gekeler1,gekeler2}.

In the particular case when $\mathtt{G}=\mathrm{GL}_n$ is a
general linear group,
there exists an explicit $\{P\}$-arithmetic group $N$ whose 
associated quotient $N\backslash\mathcal{B}(\mathtt{G}_{K_P})$ 
should morally classify maximal $X$-orders in the matrix 
algebra $\mathbb{M}_n(K)$, in the sense that the quotient complex has
one vertex for each conjugacy class of such orders. 
This is the normalizer in $\mathrm{GL}_n$ of a suitable
$\A$-order. The problem, however, is that the preceding quotient
is not always a simplicial complex, so part of the problem is 
finding a category where it can be properly defined. In the past,
we have indeed introduced some modifications to the category of graph
in which we work through the inclusion of half edges. This gets cumbersome
in higher dimensions, where the equivalents of half edges quickly explode
in number. We solve this in \S \ref{SSsing}
by introducing a suitable new category.

Following the same philosophy as in Th. \ref{MThB} above, 
we describe the cusps of the quotient complex in terms
of some natural subspaces that accept a description in 
terms of quotient of
apartments by some natural permutation group actions. 
For $n=2$, these spaces
are either lines or rays. In higher dimensions,
these ``quotient apartments'' can be described
as the largest subcomplexes of the quotient that 
correspond to some explicit 
families of split maximal orders. Then the 
unbranched sectors in Th. 
\ref{MthE}, i.e., the cuspidal sectors, can be 
precisely described in terms of the corresponding split
vector bundles, extending the result of Th. \ref{MThB} to the 
higher rank context. 

We expect that the tools developed here will inspire future studies on bundles
over curves, when the former are endowed with quadratic, symplectic, or Hermitian forms,
or possibly also on sheaves of $\mathtt{G}$-bundles 
(see \cite{Heinloth,KuznetsovQuadBund}). 
The automorphism groups of such forms over $K$ are classical, 
quasi-split or split, algebraic groups $\mathtt{G}$. 
The corresponding S-complexes $\mathtt{G}(\A)\backslash 
\mathcal{B}(\mathtt{G})$ have been 
studied in the previously cited works. 
Since the normalizer of an order with an involution 
or a form need not preserve vertex types (and very seldom does), the singular-complex 
approach introduced here should extend naturally to this setting.
We expect an analog of Th.~\ref{mainth1} holds, with the symmetric
group $S_n$ replaced by the Weyl group $W(\mathtt{G})$ acting on split apartments, and
an analog of Th.~\ref{mainth3} relating the resulting quotients to moduli of vector
bundles with orthogonal, symplectic, or Hermitian structure.
 The extension of Ths.~\ref{mainth-1} and~\ref{mainth1} 
 will possibly be straightforward.
Moreover, we believe that an analog of Th.~\ref{mainth5} 
holds in that setting, if the integer $n-1$ is  replaced by 
the (split) rank of the group $\mathtt{G}$. We plan to pursue 
this idea, along with its consequences, in future works.

\section{Conventions on lattices and orders}\label{section conventions}

For any abelian group $C$, we consider the
diagonal $\Delta_C=\{(a,\dots,a)|a\in C\}
\subseteq C^n$, and we let $\nabla_C=C^n/\Delta_C$.
We also identify $C^n$ with $\mathbb{Z}^n\otimes_{\mathbb{Z}}C$,
and write expressions like $(x_1, \cdots, x_n)a$ instead of either
$(x_1, \cdots, x_n) \otimes a$ or $(x_1 a, \cdots, x_n a)$ when convenient.
This is often applied to the vector
$\vec{\delta}=(1,\dots,1)$, so that we write 
$\Delta_C=\vec{\delta}C$. 

Let $X$ be a smooth, projective, geometrically integral curve defined 
over a field $k$. Let $P$ be 
a closed point in $X$. Let $K =K(X)$ be the field of rational functions 
on $X$, let $K_P$ be the completion at $P$ of $K$.
Let $\mathcal{O}_P$ be the ring of integers of $K_P$, and let us fix 
a uniformizer $\pi_P\in \mathcal{O}_P$.
The residue field $k(P)=\oo_P/\pi_P\oo_P$ at $P$ 
is a finite extension of $k$.
Note that $U_P = X \smallsetminus\{P\}$ is an affine open subset  of $X$.
We denote by $\A = \oo_X(U_P)$ the ring of rational functions on $X$ that 
are regular on $U_P$.

By an $X$-lattice in a finite dimensional
$K$-vector space $V$, we mean a sheaf of 
$\mathcal{O}_X$-modules whose generic fiber is $V$.
The term $X$-bundle is also used in literature.
For any sheaf of groups $\Lambda$ on $X$ we 
denote by $\Lambda(U)$ its group of $U$-sections. 
In particular, this convention applies to $X$-
lattices and $X$-orders. Note that, for every 
affine open subset $U\subseteq X$,
the $\mathcal{O}_X(U)$-module $\Lambda(U)$ is an 
$\mathcal{O}_X(U)$-lattice. If we denote
an $X$-lattice by $\Lambda$, then we write $\Lambda_Q\subseteq V_Q$ for its completion
at $Q$, for any closed point $Q\in X$. This is, by definition, the completion
at $Q$ of $\Lambda(U)$ for any affine open set $U$ containing $Q$. 
Completions satisfy the following properties, which, together, are called
the local-global principle for lattices
\cite[\S 81:14]{omeara}:
\begin{enumerate}[label={(LG\arabic*{})}]
    \item\label{LGdet} An $X$-lattice is completely determined by the family 
    $\{\Lambda_Q\}_Q$ of such completions.
    \item\label{LGequals} For any pair of $X$-lattices, $\Lambda$ and $\Lambda'$, the families 
    $\{\Lambda_Q\}_Q$ and $\{\Lambda'_Q\}_Q$ are coherent in the sense that 
    $\Lambda_Q=\Lambda'_Q$ at all but a finite number of points $Q$.
    \item\label{LGvariant} There exists an $X$-lattice
    corresponding to every prescribed family that is 
    coherent to $\{\Lambda_Q\}_Q$ in this sense.
\end{enumerate}
The latter property allows us to define
new lattices by modifying a finite number of 
completions of an existing 
lattice. 

When $V$ is a $K$-algebra, a lattice 
$\mathfrak{H}$ is called an $X$-order if
$\mathfrak{H}(U)$ is a subring with $1$ 
for any open set $U$, or, equivalently,
if $\mathfrak{H}_Q$ is a ring with $1$ for any 
closed point $Q$. A central simple 
algebra, like a matrix algebra, always has
maximal $X$-orders. This is proven in \cite[Cor. 10.4]{Reiner}
for Dedekind domains, and can be easily extended to
our setting by using property \ref{LGvariant} above.
Since an order $\DD$  in an 
algebra $\mathfrak{A}$ is a lattice
satisfying both $1_{\mathfrak{A}}\in\DD(X)$ and 
$\DD\DD=\DD$, the analog for orders
of the local-global principle in
\ref{LGdet}, \ref{LGequals} and \ref{LGvariant}
also holds.

All the preceding conventions have analogs for
lattices or orders defined over any open set 
$U\subseteq X$. We often apply them to 
$\A$-lattices, where $\A=\mathcal{O}_X(U_P)$ as
before. We also have local analogs.
Furthermore, every maximal order in 
$\mathbb{M}_n(K_P)$ is the endomorphism ring 
$\mathrm{End}(L)$ of an $\mathcal{O}_P$-lattice
$L$ in $K_P^n$, and $\mathrm{End}(L)=
\mathrm{End}(L')$ if and only if $L'=IL$ 
where $I$ is a (local) principal ideal.
The global version of this property states that
every maximal order in $\mathbb{M}_n(K)$ is the 
sheaf of endomorphism $\mathfrak{End}(\Lambda)$
of an $X$-lattice in $K^n$. This order is denoted
$\DD_\Lambda$ in the sequel. Furthermore
$\DD_\Lambda=\DD_{\Lambda'}$ if and only if 
$\Lambda'=\LL\Lambda$
for some rank one $X$-lattice (or sheaf of fractional 
ideals) $\LL$. Note that the order $\DD_\Lambda$ is
defined on every open set 
$U\subseteq X$ by
$$\mathfrak{D}_\Lambda(U)=\left\{a\in 
\mathbb{M}_n(K) 
\Big|a\Lambda(U)\subseteq\Lambda(U)\right\}.$$
For any basis $\beta=\{\vec{v}_1, \ldots, \vec{v}_n\}$ of $K_P^n$, i.e., a
local basis, a $\beta$-split local lattice is a module of the form
\begin{equation}\label{eq lattice}
L=L_{\vec u,P,\beta}= \pi_P^{u_1}\mathcal{O}_P\vec{v}_1 +\cdots + 
\pi_P^{u_n}\mathcal{O}_P\vec{v}_n,
\textnormal{ where }\vec u=(u_1,\dots,u_n)\in\mathbb{Z}^n. 
\end{equation}
Replacing $L$ by another lattice in the same homothety class
has the effect of replacing $\vec u$ by $\vec u+m\vec \delta$
where $\vec \delta=(1,\dots,1)$.
Thus, we can always assume $u_1=0$
by this type of replacement.

A divisor of $X$ is a formal sum of the form $B=\sum_{P \in X} n_P P$, where 
$P$ is a closed point of $X$ and $n_P \in \mathbb{Z}$.
The valuation $\nu_P(B)$ at $P$ of $B$ is the coefficient $n_P$.
The set of divisors forms a group with the usual addition, which we denote by 
$\mathrm{Div}(X)$ in the sequel.
The divisor $B$ above is effective when 
every coefficient is non-negative. We write $B'\leq B$ when $B-B'$ is 
effective.
For any divisor $B$, the associated sheaf of fractional ideals (also called a 
line bundle) is defined, at any open set $U \subseteq X$, by the formula:
$$\LL^{B}(U)=\lbrace f \in K \big{\vert} (\mathrm{div}(f)+B) \vert_U 
\geq 0 \rbrace.$$
When $\beta=\{\vec{v}_1, \ldots, \vec{v}_n\}$ is a $K$-basis of $K^n$, 
called a global basis in the sequel, a $\beta$-split $X$-lattice is a lattice of the form 
$$\LL_{\BB,\beta}= \LL^{B_1}\vec{v}_1 +\cdots + \LL^{B_n}\vec{v}_n,$$
for some vector of divisors $\BB=(B_1,\dots,B_n)\in \text{Div}(X)^n$.
Note that the completion at $P$ of $\LL_{\BB,\beta}$ is
$L_{\vec b,P,\beta}$, where 
$\vec b=\big(\nu_P(B_1),\dots,\nu_P(B_n)\big)$.
A split $X$-lattice is a $\beta$-split $X$-lattice for some 
(often unspecified) basis $\beta$. 

We say that a maximal $X$-order is split whenever it is the
endomorphism ring of a split $X$-lattice.
We can also specify the
basis and call it a $\beta$-split maximal $X$-order.
To state this more precisely, we use the following convention:
For any vector $\vec u$ we write $u_\beta$ for the column matrix whose
entries are the coefficients of $\vec u$ in the basis $\beta$.
In particular, if $\epsilon$ is the canonical basis, then
the entries of $u_\epsilon$ are just the coordinates of $\vec u$.
Then, a $\beta$-split maximal $X$-order has the form $\DD_{\BB,\beta}= 
\mathtt{P}^{-1}\left(\LL^{B_i-B_j}\right)_{ij}\mathtt{P}$,
for the matrix $\mathtt{P} \in \mathrm{GL}_n(K)$ satisfying $\mathtt{P}u_{\epsilon}=u_\beta$ for every vector $\vec u$.

In the sequel, we write $[\BB]$ for the class of a divisor vector 
$\BB$ in the quotient group $\nabla_{\mathrm{Div}(X)}=
\mathrm{Div}(X)^n/\Delta_{\mathrm{Div}(X)}$.
Since $\DD_{\BB+\vec \delta B',\beta}=\DD_{\BB,\beta}$ for any divisor 
$B'\in\mathrm{Div}(X)$, we usually write $\DD_{[\BB],\beta}$ instead.
The notation $[[B_2,\dots,B_n]]=[(0,B_2,\dots,B_n)]$ is also used.
Note that $[(B_1,B_2,\dots,B_n)]=[[B_2-B_1,\dots,B_n-B_1]]$.

The group $\mathbb{K}=(K^*)^n$ acts on the group $\mathrm{Div}(X)^{n}$
of divisor vectors
via $\vec F*\BB=\BB+\mathop{\mathrm{div}}(\vec F)$, where 
$\mathop{\mathrm{div}}(\vec F)=
\big(\mathop{\mathrm{div}}(f_1),\dots,\mathop{\mathrm{div}}(f_n)\big)$,
whenever $\vec F=(f_1,\dots,f_n)$. On the other hand, 
$\mathbb{K}$ also acts
on the set of ordered bases via $\vec F(\vec{v}_1,\dots,\vec{v}_n)=
(f_1\vec{v}_1,\dots,f_n\vec{v}_n)$. With these definitions, it is 
straightforward that $\DD_{\vec F*\BB,\beta}=\DD_{\BB,\vec F\beta}$.
We conclude that the isomorphism class of $\DD_{\BB,\beta}$
depends only on the image of $\BB$ in the quotient group
$\mathop{\mathrm{Pic}}(X)^n\cong \mathrm{Div}(X)^{n}/\mathbb{K}$, or, considering the discussion in the 
preceding paragraph, even the quotient
$\nabla_{\mathrm{Pic}(X)}$.
For this reason, if $\langle\BB\rangle$ denotes the class of 
$\BB$ in $\nabla_{\mathrm{Pic}(X)}$,
then the order $\DD_{\BB,\beta}$ is often written as 
$\DD_{\langle\BB\rangle}$. Note, however, that 
$\DD_{\langle\BB\rangle}$ is only defined as an abstract
order, or equivalently, as an isomorphism class of orders
in $\mathbb{M}_n(K)$.

\section{Main results}\label{section main results}

Let $X$ be a smooth, projective, geometrically integral curve defined over a 
field $k$, and let $K=k(X)$ be its function field.
In this article, we study a quotient of the Bruhat-Tits building of
$\mathrm{SL}_n(K_P)$ that classifies maximal $X$-orders on a
matrix $K$-algebra $\mathbb{M}_n(K)$.
This quotient is a higher dimensional version of the C-graph 
described in \S \ref{section intro}.
Let $D$ be a maximal $\A$-order in $\mathbb{M}_n(K)$, where $
\A = \mathcal{O}_X(U_P)$. The Bruhat-Tits building 
$\mathcal{B}_P(D)$ is a simplicial complex whose vertices are 
the maximal $X$-orders in $\mathbb{M}_n(K)$ having
$D$ as its restriction to $U_P$.
The general linear group $\mathrm{GL}_n(K)$ acts on the set of
maximal $\A$-orders by conjugation, whence the normalizer 
$N = \mathrm{Norm}_{\mathrm{GL}_n(K)}(D)$ acts on the set of 
vertices of $\mathcal{B}_P(D)$. This action extends naturally to 
an action of $N$ on the building itself. 
This suggests that such action can be used
to study the group $N$ or some of its sub-groups of finite  
index. It helps, therefore, to be able to describe the quotient
of the building corresponding to this action.

There is
a technical problem to make sense of this statement,
as the quotient building is not well defined. We have overcome
this problem in the past, in the $\mathrm{SL}_2$ case, where 
buildings are graphs, by passing to the barycentric subdivision.
This can be achieved in the general case, in principle, but
the required definitions become too cumbersome to be practical.
Here we use an alternate approach that regards a building as
a topological space with an associated singular chain complex.
See \S \ref{SSsing} and \S 
\ref{section Quotient Buildings and conjugacy classes} 
for details. Singular simplices can have multiple edges
or faces, so quotients can be defined with no assumption on
the action, other than being simplicial. These quotients 
fall in a category whose objects are called singular
complexes.

We say that two $X$-orders $\DD$ and $\DD'$ 
belong to the same genus whenever they are locally conjugate at all places, 
or equivalently, when they are adelically conjugate
(see \S \ref{section conventions} and
\S\ref{section Quotient Buildings and conjugacy classes}
respectively). Furthermore, we say $\DD$ and 
$\DD'$ are in the same spinor genus if there exist a 
global matrix $g \in \mathrm{GL}_n(K)$ and an adelic matrix $a$ 
with trivial determinant such that $\DD' = g a 
\DD a^{-1} g^{-1}$. An equivalence class under this 
relation is called a spinor genus.
This notion refines the genus classification: while the usual 
genus describes a local conjugacy condition, spinor equivalence 
introduces a finer global distinction, encoded in the quotient
group $\mathrm{Pic}(X)/\mathrm{Pic}(X)^n$.
Indeed, over affine open sets, spinor genera and conjugacy classes
coincide. Now we define the C-complex as the singular complex
$\mathscr{C}_P(D) = N \backslash \mathscr{B}_P(D)$, where $\mathscr{B}_P(D)$ is the singular complex associated to 
$\mathcal{B}_P(D)$. This is the right quotient to study orders, 
since we have the following theorem:

\begin{theo}\label{mainth0}
The vertex set of the C-complex $\mathscr{C}_P(D)$ is in bijection with 
the set of conjugacy classes of maximal $X$-orders in $\mathbb{M}_n(K)$ whose 
restrictions to $U_P$ are conjugate of $D$. Moreover, if $k$ is finite,
these are the conjugacy classes in precisely $d$
spinor genera of maximal $X$-orders, where $d$ is the order
of the image of $P$ in $\mathrm{Pic}(X)/\mathrm{Pic}(X)^n$. 
\end{theo}

\begin{ex}
If $X=\mathbb{P}^1$, then the integer $d$ in Th. \ref{mainth0} is $\frac nc$,
where $c$ is the greatest common divisor 
between $n$ and the degree of $P$, since $\mathrm{Pic}(X) \cong\ZZ$.
\end{ex}

An $m$-dimensional complex, for an integer $m\geq1$, is called combinatorially 
finite if it is the union of a finite number of cuspidal sector chambers 
and an $r$-neighborhood, for $r\geq 0$, of a combinatorially 
finite $(m-1)$-dimensional complex. A $0$-dimensional
combinatorially finite complex is a finite set of vertices.
With this definition we can give the following structural result 
for the C-complex:

\begin{theo}\label{mainth-1}
For any maximal $\A$-order $D\subseteq\mathbb{M}_n(K)$, and for any 
place $P$, there exists a family of sector chambers 
$\mathbf{Q}_1, \cdots, \mathbf{Q}_s$ in $\mathscr{B}_P(D)$, called cuspidal 
sector chambers, such that:
\begin{itemize}
\item[(a)] the $N$-orbits of distinct sector chambers 
$\mathbf{Q}_i$ are disjoint, and each injects into 
the subjacent topological space of
the quotient $\mathscr{C}_P(D)$ via the canonical 
projection,
\item[(b)] the quotient $\mathscr{C}_P(D)$ is obtained by gluing the associated closed sector chambers $\mathscr{Q}_1, \cdots, \mathscr{Q}_s$ and
a connected singular complex $\mathscr{Y}$ along suitable subcomplexes, and
\item[(c)] the complex $\mathscr{Y}$ is an $r$-neighborhood of
an $(n-2)$-dimensional combinatorially finite complex.
\end{itemize}
In particular, $\mathscr{C}_P(D)$ is a combinatorially finite 
$(n-1)$-dimensional complex.
\end{theo}

The preceding result refines, for $N$, the theorem of Bux, Köhl and
Witzel quoted in \S\ref{section intro}. Indeed, the latter provides a covering
by $N$-orbits, while the former describes the shape of the quotient
itself. See Th.~\ref{teo cover witzel} for details.
Next result makes the description of the cusps more 
precise in terms of suitable images of certain explicit
apartments $\Sigma_{\BB,P,\beta}$ of $\mathscr{B}_P(D)$ 
in the quotient $\mathscr{C}_P(D)$, called globally $\beta$-split 
apartments. These are subcomplexes whose vertex set corresponds 
to a set of maximal $X$-orders of the form $\DD_{\vec{B},\beta}$, which 
are split by a common global basis $\beta$. See \S \ref{section BTB} 
for details. If $\beta$ is not relevant we simply say a globally split
apartment, or a GSA.

\begin{theo}\label{mainth1}
There is a finite number of GSA's $\Sigma_{\BB,P,\beta}$ in 
$\mathscr{B}_P(D)$ whose images $\hat{\Sigma}_{\BB,P,\beta}$ in 
$\mathscr{C}_P(D)$ are isomorphic to quotients 
$W\backslash\Sigma_{\BB,P,\beta}$, where $W$ is a group of affine 
transformations on $\Sigma_{\BB,P,\beta}$, regarded as an affine 
space. Furthermore, $W$ is isomorphic to a subgroup of the symmetric 
group $S_n$ that depends on the apartment. 
Each set of the form $\hat{\Sigma}_{\BB,P,\beta}$ contains 
$n!/|W|$ cuspidal sectors, and, conversely, every 
cuspidal sector is contained in a globally split apartment.
The vertices lying on the sets 
$\hat{\Sigma}_{\BB,P,\beta}$ are precisely those corresponding to 
split maximal orders.
\end{theo}

Although every split maximal order corresponds to a vertex in a 
quotient of the form $\hat{\Sigma}_{\BB,P,\beta}$, also called a
quotient split apartment or QSA, not all of them belong to a cuspidal 
sector, since they need not have flat stars. See Fig. \ref{nbh1} to see
what a flat star looks like. 

Next result connects the C-complex that we have described thus far with the S-complex 
that is most frequently treated in literature. See \S \ref{section intro}.
The latter is defined in this work
as the quotient $\mathscr{S}_P(D) =
\Gamma \backslash \mathscr{B}_P(D)$, where $\Gamma = K^* D^*$.
Note that, when $D=\mathbb{M}_n(\A)$, then 
$\Gamma=K^{*}\mathrm{GL}_n(\A)$.
Indeed, the S-complex $\mathscr{S}_P(D)$ has a 
natural moduli interpretation: its vertices are in bijection with 
the isomorphism classes of rank $n$ vector bundles 
on $X$ whose restriction to $U_P$ is isomorphic to $D = 
\DD(U_P)$. Although this might not be explicitly stated in 
this degree of generality in the literature, it follows 
from the same techniques as in \cite{SerreTrees, Stuhler} in a 
straightforward manner. Since every maximal $X$-order in 
$\mathbb{M}_n(K)$ is the endomorphism sheaf 
$\DD_\Lambda = \mathfrak{End}_{\mathcal{O}_X}(\Lambda)$ of a unique
homothety class (up to multiplication by an invertible sheaf) of
$X$-lattices $\Lambda$, as in \S\ref{section conventions}, 
the vertex set of $\mathscr{S}_P(\DD)$ also 
parametrizes $\Gamma$-orbits of maximal $X$-orders. 
Th.~\ref{mainth0} states that the C-complex 
$\mathscr{C}_P(\DD)$ provides a weaker classification.
Both complexes are related through a ramified cover,
which is a surjective and locally surjective map in the category
of singular complexes, see Def. \ref{Dram} for details. If $\phi:\mathscr{S}\twoheadrightarrow\mathscr{C}$ is a ramified cover,
a point $x$ in the realization $\mathbf{C}$ of $\mathscr{C}$ 
is unramified if $\phi$ is locally injective
at every pre-image of $x$. 

\begin{theo}\label{mainth3}
There exists an explicit ramified cover
$\phi:\mathscr{S}_P(D)\twoheadrightarrow\mathscr{C}_P(D)$,
for which every point in a cuspidal sector is unramified.
Furthermore, if $k$ is finite, the number $r$ of pre-images in $\mathscr{S}_P(D)$ of any 
vertex $v$ in $\mathscr{C}_P(D)$ is a divisor of $\frac1d|\mathrm{Pic}(X)/\mathrm{Pic}(X)^n|$,
with equality whenever the group $\mathrm{Pic}^{0}(X)$, which can be interpreted
as the group of $k$-points of the Jacobian variety, is trivial.
\end{theo}

\begin{ex}
    If $X=\mathbb{P}^1$, then $l:=|\mathrm{Pic}(X)/\mathrm{Pic}(X)^n|/d=n/d$ is
    the greatest common divisor between $n$ and the degree of $P$. In 
    particular, if the latter are coprime, then $l=1$, and therefore
    there is a bijection between the vertex set of either complex.
    Actually, in this case $N/\Gamma$ is trivial since it is isomorphic
    to the annihilator of $n$ in the abelian group $\mathrm{Pic}(\A)$, whence
    $\mathscr{C}_P(D)$ and $\mathscr{S}_P(D)$ are isomorphic.
\end{ex}

The C-complex $\mathscr{C}_P(D)$ depends only on the conjugacy 
class of the
$\A$-order $D$. Let $\Xi$ denote the orbit set $\mathrm{GL}_n(K)\backslash\Theta_{\A}$,
where $\Theta_{\A}$ is the set of maximal $\A$-orders.
This allows us to define the 
full C-complex $\mathscr{C}_P=\coprod_{[D]\in \Xi} \mathscr{C}_P(D)$,
whose vertices correspond to all conjugacy classes of maximal $X$-orders. 
In this context we have next result:

\begin{theo}\label{mainth5}
Assume that $k$ is finite. Then the total number of cuspidal sector
chambers in the full C-complex is exactly  $\vert \mathrm{Pic}(\A)
\vert^{n-1}$.
The full C-complex has 
$|\mathrm{Pic}(R)/\mathrm{Pic}(R)^n|$ connected components.
Moreover, every connected component contains
the same number of cuspidal sector chambers, namely 
$\frac{|\mathrm{Pic}(\A)|^{n-1}}{|\mathrm{Pic}(R)/\mathrm{Pic}(R)^n|}$.
\end{theo}

Finally, in \S \ref{section examples} we compute several images of GSA's in
C-complexes and describe all cusps for the C-complex of $\mathbb{P}^1_k$ 
at a place $P$ of degree $2$ or $3$.
The case $\deg(P)=1$ is a direct 
consequence of Th.~\ref{teo soule quot} since $\mathrm{Pic}(R)$ is trivial.
These results could have interesting consequences on presentations of 
certain higher-rank Hecke congruence subgroups, as 
it is shown in \cite{BravoHecke} for $n=2$.

Applying the theory of small categories without loops
(cf.~\cite[Ch. III.$\mathcal{C}$]{Bri}) to the quotient 
complexes described in Ths.~\ref{mainth-1}, \ref{mainth1}, 
\ref{mainth3} and \ref{mainth5}, we hope to obtain 
in subsequent works
decompositions of arithmetic groups such as $N$ and $\Gamma$ 
as fundamental groups of a complex of groups built from the 
cuspidal sector chambers and the connected complex $\mathscr{Y}$,
generalizing the amalgamated product decompositions of Serre and Soul\'e recalled in \S \ref{section intro}.
Some results in this direction can be found in \cite[\S~11]{BravoLoisel}.
These decompositions would have natural consequences on the (co)homology of $N$ and $\Gamma$,
in the spirit of the works of Soul\'e~\cite[\S 2, Th. 3]{So}, Harder~\cite{H2},
Serre~\cite[Ch. II, \S 2]{SerreTrees}, Margaux~\cite{Margaux}, as well as 
Knudson and Wendt~\cite{Knudson1,Knudson2,KnudsonBook,Wendt}.

\section{The local building and its global 
interpretation}\label{section BTB}
The purpose of this section is to give a precise
definition of the Bruhat-Tits building $\mathcal{B}_P(D)$ 
introduced in \S \ref{section main results}.
In the next section we turn this complex into an object 
$\mathscr{B}_P(D)$ in a category where quotients can 
always be defined. For our purposes, we define the 
local Bruhat-Tits building $\mathcal{B}_P$ as follows:
\begin{enumerate}[label=($\star$\bf{\arabic*})]
\item\label{st1} The vertices of $\mathcal{B}_P$ are the
local maximal $\mathcal{O}_P$-orders. 
\item\label{st2} A set $\{\DD_1,\ldots,\DD_{m+1}\}$
of $m+1$ orders defines an $m$-simplex (or an $m$-dimensional simplex) 
if and only if we can write $\DD_i=\DD_{L_i}$ for suitable $\mathcal{O}_P$-lattices
$L_1,\dots, L_{m+1}$ satisfying the following chain of relations:
$$\pi_PL_{m+1}\subsetneq L_1\subsetneq \cdots \subsetneq L_m
\subsetneq L_{m+1}.$$
\end{enumerate} 
Note that $L_{m+1}/\pi_PL_{m+1}\cong k(P)^n$ in \ref{st2}, so that the
preceding chain can be interpreted as a flag of subspaces of
this quotient space.
Moreover, recalling that any local maximal order comes from a local lattice of full 
rank, we can see that $\mathcal{B}_P$ is exactly the Bruhat-Tits building of $\mathrm{SL}_n(K_P)$ as defined in \cite[\S 6.9]{AbramenkoBrown}.
A simplex of $\mathcal{B}_P$ of maximal dimension is called a
chamber. A $1$-simplex is called an edge. In the notations of \S \ref{section conventions}, an apartment 
$\Sigma=\Sigma_{P,\beta}\subseteq\mathcal{B}_P$ is a subcomplex whose vertices are the maximal
orders corresponding to lattices $L$ of the form 
$L=L_{\vec u,P,\beta}$, as in Eq. \eqref{eq lattice},
for some  $\vec u=(u_1,\dots,u_n)\in \ZZ^{n}$ and
for some fixed basis $\beta$ of $K_P^n$. 
The apartment 
$\Sigma_{P,\epsilon}\subseteq\mathcal{B}_P$, defined by the canonical basis $\epsilon$, is called the standard apartment.

\begin{ex}
If $n=3$ the local maximal orders $\DD_0$, $\DD_1$ and $\DD_2$
corresponding to the lattices $L_0= \oo_P^3$, 
$L_1=\oo_P\oplus\oo_P\oplus \pi_P\oo_P$ and 
$L_2=\oo_P\oplus \pi_P\oo_P \oplus \pi_P\oo_P$ are vertices 
of a 2-simplex (chamber) inside the standard apartment.
More explicitly, they are $\DD_0= \mathbb{M}_3(\oo_P)$, 
$$\DD_1 = \begin{pmatrix}
    \oo_P & \oo_P&\pi_P^{-1}\oo_P \\
    \oo_P&\oo_P &\pi_P^{-1}\oo_P\\ 
    \pi_P\oo_P& \pi_P\oo_P& \oo_P
\end{pmatrix} \text{ and } \quad  \DD_2 = \begin{pmatrix}
    \oo_P & \pi_P^{-1}\oo_P&\pi_P^{-1}\oo_P \\
    \pi_P\oo_P&\oo_P &\oo_P\\ 
    \pi_P\oo_P& \oo_P& \oo_P
\end{pmatrix}.$$
These orders correspond to all vertices of one chamber.
\end{ex}

It is a consequence of \ref{LGvariant} that we can identify an $X$-order
$\DD$ in $\mathbb{M}_n(K)$ with the pair $\big(\DD_P,\DD(U_P)\big)$, whose first coordinate is a
local order, and whose second coordinate is an $\A$-order.
By fixing an $\A$-order $D$, we obtain a natural correspondence 
between the vertices in $\mathcal{B}_P$ and the $X$-orders 
$\DD$ satisfying $\DD(U_P)=D$. Two orders
with the same restriction to $U_P$ are called $P$-variants.
By replacing vertices via this correspondence,
we define a new building whose vertices are all $P$-variants of any
of them. It is
denoted $\mathcal{B}_P(D)$, where
$D=\DD(U_P)$, or $\mathcal{B}_P(L)$,
if $L=\LL(U_P)$ denotes one lattice in the associated
homothety class.

\begin{ex}
For instance, if $\DD=\DD_{\BB,\epsilon}$, where $\epsilon$ is the canonical
basis of $K^n$, then $D=\DD(U_P)=
\left(\LL^{B_i-B_j}(U_P)\right)_{ij}$. This is the case for the order 
$\DD_{\vec 0,\epsilon}=\mathbb{M}_n(\oo_X)$.
\end{ex}

\begin{Def}\label{d42}
For a global basis
$\beta$ of $K^n$, a globally $\beta$-split apartment
$\Sigma=\Sigma_{\BB,P,\beta}\subseteq\mathcal{B}_P(D)$ is the
largest subcomplex whose 
vertices are the maximal orders $\DD_{\BB+\vec uP,\beta}$ corresponding to lattices 
$\LL$ of the form $\LL=
\LL_{\BB+\vec uP,\beta}$, for some  
$\vec u=(u_1,\dots,u_n)\in \ZZ^{n}$.  
\end{Def}

\begin{ex}
The natural correspondence 
between $\mathcal{B}_P$ and $\mathcal{B}_P(D)$
takes the standard apartment $\Sigma_{P,\epsilon}$ to the globally split
apartment $\Sigma_{\BB,P,\epsilon}$. 
\end{ex}

For a given choice of the lattice $L=\LL(U_P)$ and the basis
$\beta$, globally $\beta$-split apartments exist if and only if
the basis $\beta$ splits $L$, i.e., it splits its completion $L_Q$
at every closed place $Q\in U_P$. However, as $\mathcal{O}_X(U_P)$
is a Dedekind domain, for any choice of $L$ (or its associated 
order), there exists some global basis splitting it, and therefore
the building always has globally split apartments. 

\begin{lem}
The apartment
$\Sigma_{\BB,P,\beta}$ depends only on the class of $\BB$ in 
$\nabla_{\mathrm{Div}(\A)}$, 
where $\A$ is the coordinate ring of the affine 
curve $X\smallsetminus\{P\}$, $\mathrm{Div}(\A)$ is the 
ideal group of the Dedekind domain $\A$,
and the notation $\nabla_C$ is explained in \S\ref{section conventions}.
\end{lem}

\begin{proof}
    We need to prove that $\Sigma_{\BB,P,\beta}=\Sigma_{\BB',P,\beta}$
    whenever $\BB'=\BB+\vec uP+\vec \delta B_0$, for a vector
    $\vec u\in\ZZ^n$ and a divisor $B_0$. 
    The summand $\vec \delta B_0$
    has no effect on the apartment, since  $\DD_{\BB_1,\beta}=
    \DD_{\BB_1+\vec \delta B_0,\beta}$, for any divisor vector $\BB_1$.
    The summand $\vec uP$
    has no effect since it preserves the set of lattices considered
    in Def. \ref{d42}.  
\end{proof}

\begin{ex}\label{ex GB}
If $X=\mathbb{P}^1$ then, by Birkhoff-Grothendieck Theorem 
\cite[Th. 2.1]{GA57}, any maximal $X$-order is $\beta$-split for 
some basis $\beta$ of $K^n$, so every vertex in $\mathcal{B}_P(D)$ 
belongs to a split apartment in this case.
\end{ex}

\begin{Remark}
    For a fixed basis $\beta$, any $\beta$-split apartment has a unique 
    maximal order $\DD_{\BB,\beta}$ with $\BB$ supported away from $P$. 
    This vertex is not special, as it might seem, since 
    $\DD_{\vec F*\BB,\beta}=\DD_{\BB,\vec F\beta}$, so its position can be 
    shifted arbitrarily by modifying the basis through this action.
\end{Remark}

\begin{Not}
For a vector $\vec u=(u_1,\dots,u_n)\in\ZZ^n$, we write 
$[[u_2-u_1,\dots,u_n-u_1]]$ for its class in $\nabla_{\ZZ}$.
This allows us to write $[[u_2-u_1,\dots,u_n-u_1]]P$ for a suitable element
in $\nabla_{\mathrm{Div}(X)}$. For the labeling of pictures in subsequent examples, we often use $[[a_2,\dots,a_n]]$ to indicate the vertex
corresponding to the order $\DD_{\BB+(0,a_2,\dots,a_n)P,\beta}$, where
the representative $\BB$ is chosen to be supported away from $P$.
However, we should keep in mind the preceding remark.  
\end{Not}

\begin{ex}
    If $n=3$, then the vertex  
    $\DD=\DD_{[[0,0]], \epsilon}$ has six neighboring vertices contained in 
    the standard apartment $\Sigma_{\vec0,P,\epsilon}$. Explicitly, 
    they are of the form $\DD_{[[m_1,m_2]]P, \epsilon}$ with 
    $[[m_1,m_2]]\in \{[[\pm 1,0]], [[0,\pm 1]], \pm[[1,1]]\}$ 
    as Fig. \ref{nbh1} shows. 
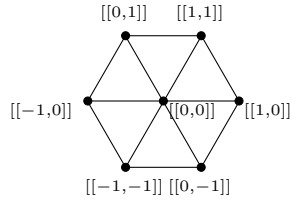
\begin{figure}[h]
\begin{center}
\begin{tikzpicture}[scale=1]
\draw (1,0)--(2,0); \filldraw (1,0) circle(1.5pt); \filldraw (2,0) circle(1.5pt);
\draw (1,0) -- ++(60:2);\draw (2,0) -- ++(120:2);
\draw (1,0) -- ++(120:1);\draw (2,0) -- ++(60:1);
\filldraw (1.5,0.877) circle(1.5pt);
\draw (0.5,0.877) -- ++(0:2);
\filldraw (0.5,0.877) circle(1.5pt);\filldraw (2.5,0.877) circle(1.5pt);
\draw (0.5,0.877) -- ++(60:1);\draw (2.5,0.877) -- ++(120:1);
\draw (1,1.74)-- ++(0:1); \filldraw (1,1.74) circle(1.5pt); 
\filldraw (2,1.74) circle(1.5pt);
\node [above] at (1.9,0.5) {${}_{{}_{[[0,0]]}}$};
\node [above] at (2.9,0.5) {${}_{{}_{[[1,0]]}}$};
\node [above] at (-0.1,0.5) {${}_{{}_{[[-1,0]]}}$};
\node [above] at (1,-0.5) {${}_{{}_{[[-1,-1]]}}$};
\node [above] at (2,-0.5) {${}_{{}_{[[0,-1]]}}$};
\node [above] at (1,1.8) {${}_{{}_{[[0,1]]}}$};
\node [above] at (2,1.8) {${}_{{}_{[[1,1]]}}$};
\end{tikzpicture}
\end{center}
\caption{The neighbors of the order $\DD_{0,\epsilon}$
lying on the standard apartment.
Note that the label $[[a,b]]$ corresponds to the order 
$\DD_{[[a,b]]P,\epsilon}$.}\label{nbh1}
\end{figure}
\end{ex}

\section{The category of singular complexes}\label{SSsing}

We can associate to every simplicial complex $\mathcal{B}$, a
topological space $\mathbf{B}$ called its realization (also 
called the polyhedron of $\mathcal{B}$ in current literature) by gluing 
together geometrical simplices. The standard geometrical $n$-simplex 
$\mathbf{T}_n\subseteq\mathbb{R}^{n+1}$ is the convex hull of the basis 
vectors $\vec e_1,\dots,\vec e_{n+1}$. One takes a copy of $\mathbf{T}_n$
for any $n$-dimensional simplex of $\mathbf{B}$, for every integer $n\geq0$.
For any subset $\{i_1,\dots,i_{m+1}\}$ of $\{1,\dots,n+1\}$, with
$i_1<\cdots<i_{m+1}$, there is a linear map sending $\vec e_j\in\mathbb{R}^{m+1}$
to $\vec e_{i_j}\in\mathbb{R}^{n+1}$ called a face map. 
Its image is called a face, and a proper
face when $m<n$. The complement of the union of the proper faces is
called the interior of the simplex and denoted $\mathbf{T}_n^o$.
The face maps are used to glue the geometrical simplex corresponding 
to each simplex of $\mathcal{B}$ with
the geometrical simplices corresponding to its faces. The quotient space
obtained from these identifications is the realization $\mathbf{B}$.
When a group $G$ acts on a complex $\mathcal{B}$, 
the quotient topological space 
$G\backslash\mathbf{B}$ is defined even when the quotient complex is not. 
Quotient complexes can be defined for arbitrary simplicial actions by
passing to the second barycentric subdivision of the complex. However,
for our purposes, it is sufficient to work with the topological space,
together with some combinatorial information. The tools to do this are 
introduced in this section.

Recall that a singular $n$-dimensional simplex in a topological 
space $X$ is a continuous function $\sigma$ from a standard simplex 
$\mathbf{T_n}$ to $X$, as in \cite[\S 2.1]{Hatcher},
while a face is a composition $\sigma\circ\phi$ where 
$\phi$ is any of the face maps
$\phi:\mathbf{T}_k\rightarrow\mathbf{T}_n$, with $k<n$,
as defined above. Permutations of a singular
simplex are defined analogously. The interior of
a singular simplex $\sigma:\mathbf{T}_n\rightarrow\mathbf{X}$
is just the set $\sigma(\mathbf{T}_n^o)$. The star of $\sigma$
is the union of the interiors of all simplices having $\sigma$ as 
a face. A singular complex is a pair
$\mathscr{X}=(\mathcal{X},\mathbf{X})$, where $\mathbf{X}$ is
a topological space, while $\mathcal{X}$ is a set of 
singular simplices in $\mathbf{X}$ satisfying the following 
hypotheses:
\begin{enumerate}
    \item Any face of a simplex in $\mathcal{X}$ is also
    a simplex in $\mathcal{X}$.
    \item Any permutation of a simplex in $\mathcal{X}$ is 
    also a simplex in $\mathcal{X}$.
    \item Any point in $\mathbf{X}$ is in the interior of a unique 
    simplex in $\mathcal{X}$, up to permutation.
    \item The star of any simplex is open in $\mathbf{X}$.
\end{enumerate}
Note that we do not assume that singular simplices are 
injective. This gives us freedom to consider simplices
with repeated faces, and also ``folded simplices'', i.e.,
the quotient of a simplex by the natural extension of an
action by permutations on its vertices.  See Fig.
\ref{fig1new} for a few examples of this phenomenon.
\begin{figure}[h]
\centering
% Define a picture box size (width, height) in terms of unit length
\setlength{\unitlength}{.6cm}
\begin{picture}(5,3.5)
  \put(1.85,2.65){\makebox(0,0){${}^\spadesuit$}}
  \put(0,0){\makebox(0,0){$\bullet$}}
  \put(3.66,0){\makebox(0,0){$\bullet$}}
    \thicklines
  \put(0,0){\line(1,0){3.66}}\put(0,0){\line(2,3){1.803}}
  \put(3.66,0){\line(-2,3){1.803}}
\end{picture}
\setlength{\unitlength}{.4cm}
\begin{picture}(6,4)
\put(2,4.3){\line(-2,-3){1.9}}
\put(2,4.3){\line(2,-3){1.9}}
\put(2,4.15){\makebox(0,0){${}^\spadesuit$}}
\put(2,0){\makebox(0,0){$\bullet$}}
  \qbezier(0,1)(0,2)(2,2) \qbezier(4,1)(4,2)(2,2)
  \thicklines
\qbezier(0,1)(0,0)(2,0)\qbezier(4,1)(4,0)(2,0)
\qbezier(1.47,1.97)(2,2.05)(2.53,1.97)
  \qbezier(2,4.3)(1,1.5)(2,0)
  \qbezier(2,4.3)(3,1.5)(2,0)
\end{picture}
\qquad\qquad
\setlength{\unitlength}{.6cm}
\begin{picture}(5,3.5)
  \put(1.8,2.75){\makebox(0,0){$\bullet$}}
  \put(1.6,2.4){\makebox(0,0){$*$}}
  \put(0,0){\makebox(0,0){$\bullet$}}
  \put(0.4,0){\makebox(0,0){$*$}}
  \put(3.66,0){\makebox(0,0){$\bullet$}}
  \put(3.4,0.35){\makebox(0,0){$*$}}
  \put(1.85,1){\makebox(0,0){${}^\clubsuit$}}
    \thicklines
  \put(0,0){\line(1,0){3.66}}\put(0,0){\line(2,3){1.803}}
  \put(3.66,0){\line(-2,3){1.803}}
\end{picture}
\setlength{\unitlength}{.4cm}
\begin{picture}(6,4)
\put(2,2.8){\line(-3,-2){1.7}}
\put(2,2.8){\line(3,-2){1.7}}
\put(2,2.75){\makebox(0,0){${}^\clubsuit$}}
\put(2,0){\makebox(0,0){$\bullet$}}
\put(2.6,0.03){\makebox(0,0){$*$}}
  \qbezier(0,1)(0,2)(2,2) \qbezier(4,1)(4,2)(2,2)
  \thicklines
\qbezier(0,1)(0,0)(2,0)\qbezier(4,1)(4,0)(2,0)
\end{picture}
\caption{Two examples of folded simplices that can appear
in a complex of $\mathfrak{sing}$. The first one is obtained 
by identifying two vertices. The second is the quotient by a 
120-degrees rotation.}\label{fig1new}
\end{figure}
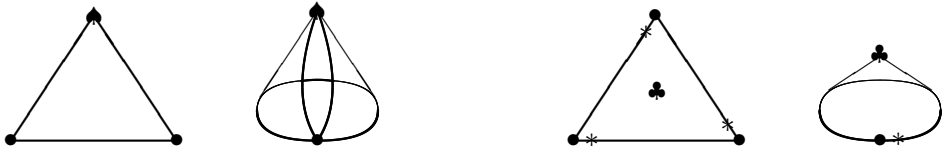
The topological space $\mathbf{X}$ is called the subjacent topological
space, or the realization, of the complex. A continuous map 
$\tau:\mathbf{X}\rightarrow\mathbf{X}'$ between 
two such realizations, is called simplicial if, for every 
singular simplex $\sigma$ in $\mathcal{X}$, the composition
$\tau\circ\sigma$ is a singular simplex in the corresponding
set $\mathcal{X}'$ for $\mathscr{X}'=(\mathcal{X}',\mathbf{X}')$. 
Note that the image of a 
singular simplex $\sigma$, under this rule, has the same dimension as 
$\sigma$, so the third condition in the definition of a singular complex
prevents us from having simplicial maps collapsing simplices. This is important for the 
definition of ramified covers. There is a category
$\mathfrak{sing}$ whose objects are singular complexes, and 
its morphisms are simplicial maps. 
For any set $\mathcal{Y}\subseteq
\mathcal{X}$ satisfying conditions $1$ and $2$ in the preceding definition,
there is an object $\mathscr{Y}$ in $\mathfrak{sing}$ whose realization is a subspace
$\mathbf{Y}$ of $\mathbf{X}$. This is called a subcomplex. In particular, the
image of a subcomplex under a simplicial map is also a subcomplex.
The advantage of working in 
this category is that any group action by simplicial maps on
a singular complex defines a quotient in the same category.
On the other hand, every simplicial complex defines a
singular complex $\mathscr{X}$ whose associated topological
space $\mathbf{X}$ is just the realization of the complex,
as defined at the beginning of this section.

Assume $\mathbf{B}$ is the realization of a simplicial complex
$\mathcal{B}$. For any simplex $\mathbf{T}$ of $\mathcal{B}$ there is
a canonical embedding $\phi_\mathbf{T}:\mathbf{T}_n\rightarrow\mathbf{B}$
that arises from the definition of $\mathbf{B}$ as a quotient space
of a disjoint union of simplices. Then a singular simplex 
$\sigma:\mathbf{T}_n\rightarrow\mathbf{B}$ is called combinatorial 
if it has the form $\phi_{\mathbf{T}}\circ\lambda$ for some simplex 
$\mathbf{T}$ of $\mathcal{B}$, where 
$\lambda:\mathbf{T}_n\rightarrow\mathbf{T}_n$ is induced by
a permutation of the vertices. The singular complex induced
by $\mathcal{B}$ is $\mathscr{B}=
(\widehat{\mathcal{B}},\mathbf{B})$, where
$\widehat{\mathcal{B}}$ is the collection of combinatorial
singular simplices for $\mathcal{B}$.
We specialize this to the case where $\mathcal{B}
=\mathcal{B}_P$ is the local Bruhat-Tits building, as defined 
in the previous section.

\begin{prop}
If a group $G$ acts by simplicial maps on a singular complex
$\mathscr{X}=(\mathcal{X},\mathbf{X})$, then the pair
$G\backslash\mathscr{X}=(G\backslash\mathcal{X},G\backslash
\mathbf{X})$, where $G\backslash\mathcal{X}$ is the set of all
compositions $\rho\circ\sigma$, where $\sigma$ is a simplex in
$\mathcal{X}$ and $\rho:\mathbf{X}\rightarrow 
G\backslash\mathbf{X}$ is the canonical projection, is a 
singular complex. 
\end{prop}

\begin{proof}
    The first two conditions in the definition of a singular 
    complex are immediate. Let $x$ be a point in $G\backslash\mathbf{X}$, and let
    $\bar{x}\in\mathbf{X}$ be any lifting. Then $\bar{x}$ is 
    in some image $\sigma(\mathbf{T}_n^o)$, whence $x\in
    \rho\circ\sigma(\mathbf{T}_n^o)$. On the other hand, if
    $x$ belongs to a second image of the form 
    $\rho\circ\sigma'(\mathbf{T}_m^o)$, then some lifting
    $\hat{x}$ of $x$ belongs to $\sigma'(\mathbf{T}_m^o)$.
    However, $\rho(\hat{x})=x=\rho(\bar{x})$ implies that
    $\hat{x}$ and $\bar{x}$ are in the same $G$-orbit, so
    $g\cdot\bar{x}=\hat{x}$ with $g\in G$, 
    whence $g\cdot\sigma$ must be a 
    permutation of $\sigma'$ by uniqueness, and therefore
    $\rho\circ\sigma$ must be a permutation of
    $\rho\circ\sigma'$. Condition 3 follows. The last condition 
    follows since the pre-image of a star in $G\backslash\mathscr{X}$
    is a union of stars in $\mathscr{X}$, and therefore it is open.
\end{proof}
\begin{Def}
The singular complex $G\backslash\mathscr{X}$ is referred to as the quotient singular complex in all that follows.
\end{Def}

\begin{ex}
    When a group $G$ acts on a simplicial complex 
    $\mathcal{B}$ in a way that the quotient complex 
    is well defined in the category of simplicial complexes, 
    then the quotient of the associated singular complex
    can be identified with the singular complex associated to the 
    quotient complex $G \backslash \mathcal{B}$. This
    is the case, for instance, if $\mathcal{B}$ is the Bruhat-Tits 
    building of $\mathrm{SL}_n(K_P)$ and $G$ is a subgroup of this special
    linear group, as it preserves vertex type, see 
    \cite[\S 6.9.2-6.9.3]{AbramenkoBrown}. However, the
    normalizer of an $\A$-order, the group $G$ that we use most 
    often here, fails to have this property in many cases, as we
    can see in the examples in \S\ref{section examples}.
\end{ex}

\begin{Def}\label{Dram}
 A simplicial map $\psi:\mathscr{S}\twoheadrightarrow\mathscr{C}$ is 
locally surjective at a simplex $\sigma$ of $\mathscr{C}$ if the star 
of any pre-image of $\sigma$ maps surjectively onto the star of $\sigma$.
A ramified cover $\psi:\mathscr{S}\twoheadrightarrow\mathscr{C}$ is a
simplicial map that is both surjective and locally surjective at every 
simplex. A point $x$ in the subjacent topological space $S$ of 
$\mathscr{S}$ is unramified if $\psi$ is injective in some neighborhood
of $x$. Note that this neighborhood need not be a star. 
\end{Def}

\begin{ex}
The simplicial map sending the real line $\mathscr{R}=
(\mathcal{R},\mathbb{R})$, with a vertex at every integer, onto the
unit circle $\mathscr{U}=(\mathcal{U},\mathbb{U})$ with a unique 
vertex at $1$, is indeed unramified at every point. See Fig. \ref{f3tg}.
\begin{figure}%[htbp]
  \centering
  \begin{tikzpicture}[
    scale=0.9,
    vertex/.style={circle, draw, fill=black, inner sep=0pt, 
    minimum size=5pt},
    edge/.style={thick},
    mapping/.style={-Stealth, line width=1.5pt, 
    shorten >= 10pt, shorten <= 10pt}
  ]
    % --- LEFT GRAPH: Infinite Path on Integers ---
    \begin{scope}[xshift=0cm]
      % Draw edges connecting integer vertices
      \draw[edge] (-3.5, 0) -- (3.5, 0);
      % Ellipses for infinity
      \node at (-4.2, 0) {$\dots$};
      \node at (4.2, 0) {$\dots$};
      % Vertices and labels at integer positions
      \foreach \x in {-3, -2, -1, 0, 1, 2, 3} {
        \node[vertex] (v\x) at (\x, 0) {};
        \node[below=4pt] at (\x, 0) {\small $v_{\x}$};}
    \end{scope}
    \node[below=4pt] at (0,1) {$\mathscr{R}$};
    \node[below=4pt] at (6,.5) {$\rightarrow$};
    % --- RIGHT GRAPH: Single Vertex with Loop ---
    \begin{scope}[xshift=8.5cm]
      % Vertex
      \node[vertex] (v) at (0, 0) {};
      % Self-loop on top
      \draw[edge] (v) to[out=45, in=135, looseness=12] (v);
      % Label for vertex
      \node[below=4pt] at (0, 0) {\small $v_0$};
      \node[below=4pt] at (.5,1) {$\mathscr{U}$};
    \end{scope}
  \end{tikzpicture}
  \caption{An example of a ramified cover that is 
  unramified at every point.}\label{f3tg}
\end{figure}
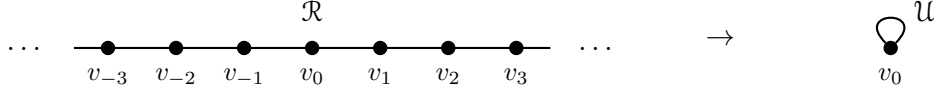
\end{ex}

\section{Quotient Buildings and conjugacy classes}\label{section Quotient Buildings and conjugacy classes}

In this section, we study the connection between vertices of a suitably
defined complex and the set of conjugacy classes of maximal $X$-orders. 
We need a couple of lemmas:

\begin{lem}\cite[Lemma 4.1]{ArenasAguiloSaavedra}\label{mainlem}
Let $G$ be a group.
Let $X$ and $Y$ be $G$-sets with $G$ acting transitively on $X$. 
Let $Z\subseteq X\times Y$
be a $G$-invariant subset. For every element $x\in X$, let $G_x=\mathrm{Stab}_G(x)$
and let $Y_x=\{y\in Y|(x,y)\in Z\}$. Then, there is a natural bijective correspondence $\phi$
between the orbit sets $G_x\backslash Y_x$ and $G\backslash Z$ sending the class 
of $y\in Y_x$ to the class of $(x,y)\in Z$.\qed
\end{lem}

In this work, we need to apply the preceding lemma in the case 
where $Y$ is the set of maximal $\mathcal{O}_P$-orders, $X$ is
a conjugacy-class of global maximal $\A$-orders, and $Z$ is the
cartesian product $X\times Y$,
which can be identified with the set of maximal $X$-orders whose
restrictions to $U=U_P$ fall in a suitable conjugacy class. In this case
$G_x$ is the normalizer of a maximal $\A$-order. In this
context, the instance of Lemma \ref{mainlem} that is relevant to us is 
the following:

\begin{lem}\label{ourmainl}
    The set of conjugacy classes of maximal $X$-orders whose restriction
    to $U_P=\mathrm{Spec}(\A)$ fall in a given conjugacy class $[D]$ are in
    correspondence with the orbits of local orders under the normalizer $N$
    of any representative $D$ of the class $[D]$. \qed
\end{lem}

Recall that the general linear group $\mathrm{GL}_n(K_P)$
acts by conjugation on the set of maximal $\mathcal{O}_P$-orders 
in the local algebra $\mathbb{M}_n(K_P)$. This can be
interpreted naturally as an action on the vertices of the local 
Bruhat-Tits building $\mathscr{B}_P$. As the action on orders 
is by conjugation, the action on vertices respects simplices
by definition \ref{st2}. It follows that the latter action 
extends to the whole building $\mathcal{B}_P$, and therefore
to the singular complex $\mathscr{B}_P=
(\hat{\mathcal{B}}_P, \mathbf{B}_P)$, as defined in the
preceding section. By applying Lemma \ref{ourmainl},
we can identify $N$-orbits of vertices with conjugacy classes of orders
whose restriction to $U$ fall in the class $[D]$. This is equivalent
to replacing $\mathscr{B}_P$ with the building $\mathscr{B}_P(D)$ whose
vertices correspond to global orders.

\begin{Not}
    The $N$-action on the topological realization 
    $\mathbf{B}_P(D)$ associated to $\mathcal{B}_P(D)$
    is denoted $x\mapsto g*x$, for $g\in N$. Then
    expressions like ``The $*$-stabilizers of a 
    subspace'' must be interpreted in terms of this action. 
    Note that this action is simplicial.
\end{Not}

\begin{Def}
    When $\mathcal{B}=\mathcal{B}_P(D)$ is the 
    Bruhat-Tits building
    defined from an $\A$-order $D$, we denote by 
    $\mathscr{B}_P(D)=
    \big(\hat{\mathcal{B}}_P(D),\mathbf{B}_P(D)\big)$ 
    the corresponding singular complex. Then we denote by 
    $\mathscr{C}_P(D)=
    N\backslash\mathscr{B}_P(D)$ the quotient by the 
    normalizer $N$ of $D$ in $\mathrm{GL}_n(K)$. This is 
    called the C-complex of $D$. This is a singular complex 
    that, in general, fails to be induced by a simplicial
    complex in the sense described in the preceding section,
    as it is apparent in the examples in 
    \cite[Ex. 5.8 \& 
    Ex. 8.1-8.3]{ARENASCARMONA2014258}. 
    In later computations
    we use also the S-complex, which is defined 
    as $\mathscr{S}_P(D)=
    \Gamma\backslash\mathscr{B}_P(D)$, for the 
    subgroup $\Gamma=K^*D^*$. Note that $\Gamma$ preserves 
    the type of vertices, whence the S-complex is indeed
    induced by a simplicial complex. 
\end{Def}

We define the ad\`ele ring $\mathbb{A}$ of $X$
as the subring of $\prod_{P \in |X|} K_P$ 
consisting of the tuples $a=(a_{P})_{P}$ 
where $a_P$ lies in the ring $\mathcal{O}_{P}$ 
for all but finitely many closed
points $P$. This ring is endowed with the 
only topology for which the product ring
$\mathcal{O}_\mathbb{A}=\prod_{P \in |X|} 
\mathcal{O}_{P}$ is an open subring.
We also define the id\`ele group $\mathbb{I}$ 
as the group of invertible ad\`eles 
$\mathbb{A}^{*}$. 
Given a finite dimensional vector space $W$ over 
$K$, we define its adelization $W_{\mathbb{A}}$ 
as $W_{\mathbb{A}}=W \otimes_{K} \mathbb{A}$.
In particular, this definition applies to 
$W=\text{End}_K(V)$.
We also define the adelization of a lattice 
$\Lambda$ by $\Lambda_{\mathbb{A}} = 
\prod_{P \in |X|} \Lambda_P$,
which is an open and compact 
subgroup of $V_{\mathbb{A}}$.
Given an arbitrary $X$-lattice $\Lambda$ and an 
adelic matrix  
$a\in \mathrm{End}_{\mathbb{A}}(V_\mathbb{A})= 
\big(\text{End}_K(V)\big)_{\mathbb{A}},$
we define the adelic image $L=a\Lambda$ of 
$\Lambda$ as the unique $X$-lattice satisfying 
$L_{\mathbb{A}}=a\Lambda_{\mathbb{A}}$.
This is possible because of Property
\ref{LGvariant} of the local-global principle 
for lattices. Since $X$-lattices are locally
free, the group $\mathrm{Aut}_{\mathbb{A}}(V_\mathbb{A})$ 
acts  transitively on the set of $X$-lattices in $V$.
In particular, we have next result:

\begin{lem}\label{lemma adeconj}
Let $\DD$ and $\DD'$ be two 
maximal $X$-orders in $\mathbb{M}_n(K)$. Then, 
there exists $a \in \mathrm{GL}_n(\mathbb{A})$ 
such that $\DD'=a \DD a^{-1}$. \qed
\end{lem}

In general, if we fix an $X$-order $\mathfrak{H}$ 
of maximal rank, then we can define the genus 
$\mathrm{gen}(\mathfrak{H})$ of $\mathfrak{H}$ as 
the set of all $X$-orders of the form $a 
\mathfrak{H} a ^{-1}$, for $a \in 
\mathrm{GL}_n(\mathbb{A})$. In this context,
the preceding lemma states that all maximal orders
belong to a single genus. The general strategies that we use 
to study maximal orders extend often to other genera.
Classifying orders in a genus can be simplified
by introducing the concept of spinor genus.

\begin{Def}\label{def spinor equivalence}
Two maximal orders $\DD$ and $\DD'$ are called spinor equivalent if there
exist a global matrix $g \in \mathrm{GL}_n(K)$ and an adelic matrix $a \in \mathrm{GL}_n(\mathbb{A})$ with determinant $1$ such 
that $\DD'=ga\DD a^{-1}g^{-1}$. An equivalence class under this relation is called 
a spinor genus. 
\end{Def}

 Over affine curves, the conjugacy classes of maximal orders can be 
 identified with the spinor genera according to the following result:

\begin{lem}\label{afinsgescl}
If $C$ is an affine curve, then two maximal orders are 
conjugate if and only if they
are spinor equivalent.
\end{lem}

\begin{proof}
It is immediate that conjugate orders are spinor
equivalent, so we prove the converse. Let $\A'=
\oo_C[C]$ be the ring of regular functions of $C$.
Assume the maximal $\A'$-orders $D$ and $D'$ are 
spinor equivalent, say $D'=gaDa^{-1}g^{-1}$ for a
global matrix $g$ and an adelic matrix $a$ with 
determinant $1$. It suffices to prove that $a$ 
can be replaced by a global matrix $h$ satisfying 
$aDa^{-1}=hDh^{-1}$.
Write $D=cD_0c^{-1}$, where $D_0=
\mathbb{M}_n(\A')$, as before. Note that the 
normalizer of $D_P$ is open locally everywhere, 
and it contains $\mathrm{GL}_n(\oo_P)$ at 
every place $P$ where 
$c_P\in\mathrm{GL}_n(\oo_P)$. We conclude that 
the set of adelic matrices $b$ satisfying 
$aDa^{-1}=bDb^{-1}$ is open in the adelic 
topology. Then the result follows by the
strong approximation theorem for 
$G=\mathrm{SL}_n$ 
\cite[Th. 7.12]{PlatonovRapinchuk}.
\end{proof}

The spinor genera of maximal orders can be determined in terms of the Picard group as follows:

\begin{lem}\label{lemma spinor genus}
Let $C$ be an affine or projective curve. The set of spinor genera of maximal
C-orders is naturally in correspondence with the quotient group
$\mathrm{Pic}(C)/\mathrm{Pic}(C)^{n}$.
\end{lem}

\begin{proof}
Let $\DD_0=\mathbb{M}_n(\oo_C)$, and 
write $\DD=b\DD_0b^{-1}$ and $\DD'=c\DD_0c^{-1}$, where $b$ and $c$ are 
adelic matrices. 
Since two maximal orders are always adelically conjugate according to Lemma \ref{lemma adeconj}, these are any two maximal orders.
It suffices to prove that they are in the same spinor genus if and only if the image of $\det(cb^{-1})$ in $\mathrm{Pic}(C)$ is an $n$-th power.
Since the local normalizer of an order $\DD_P$ is $K_P^*\DD_P^*$, and the determinant of any element in $\DD_P^*$ is a unit, the condition is independent of the choice of $b$ and $c$.
If $\DD$ and $\DD'$ are spinor equivalent, then we can assume $cb^{-1}=ga$, with $g$ and $a$ as in Def. \ref{def spinor equivalence}.
Then the determinant of $a$ is trivial and the determinant of $g$ spans 
a principal ideal, so the conclusion is immediate. 
If the image of $\det(cb^{-1})$ in $\mathrm{Pic}(C)$ is an $n$-th power, then $\det(cb^{-1})=rsh^n$ where $r\in K^*$, $s$ is a unit, locally at all places, and $h$ is an arbitrary id\`ele. We can post-multiply $b$ by an element centralizing $\DD_0$ with norm $sh^n$, and pre-multiply $c$ by a global matrix with determinant $r$, so that we can assume that $\det(cb^{-1})=1$. We conclude that $\DD$ and $\DD'$ are
spinor equivalent.
\end{proof}

\begin{proof}[Proof of Th. \ref{mainth0}]
By Lemma \ref{ourmainl}, the vertex set of $\mathscr{C}_P(D)$ is in bijection 
with the conjugacy classes of maximal orders in $\mathbb{M}_n(K)$ whose 
restrictions to $U_P$ are conjugate of $D$. Thus the first statement follows.
For the second statement, 
assume that $k$ is finite.
We need to see that there are precisely
$d$ spinor genera of maximal $X$-orders whose restriction to $U_P$ is in
the spinor genus of $D$, since two maximal $\A$-orders are conjugate if and only 
if they belong to the same spinor genus. Now the spinor genera of maximal
$X$-orders are in correspondence with the quotient 
$\mathrm{Pic}(X)/\mathrm{Pic}(X)^n$. On the other hand, the spinor genera
of maximal $\A$-orders are in correspondence with the quotient
\begin{equation}\label{fpica}
    \mathrm{Pic}(\A)/\mathrm{Pic}(\A)^n\cong\mathrm{Pic}(X)/\left(\left\langle[P]_X\right\rangle+\mathrm{Pic}(X)^n\right),
\end{equation}
where $[P]_X$ is the class of $P$, as a divisor, in $\mathrm{Pic}(X)$,
whence the result follows.
\end{proof}

Next result is useful in order to prove Th. \ref{mainth3} in 
\S \ref{ch4}. This is actually shown in the proof of 
\cite[Th. 1.2]{ArenasBordeaux}, see
\cite[p. 696]{ArenasBordeaux}. For the sake of completeness, we
prove a lemma with the precise statement that is needed here.

\begin{lem}\label{lemma N/G}
The group $N/\Gamma$ is isomorphic
to the annihilator of $n$ in the abelian group $\mathrm{Pic}(\A)$.
\end{lem}

\begin{proof}
For every $a\in N$ the principal ideal generated by
the determinant of $a$ is an $n$-th power locally at all 
places outside of $P$. Indeed, since $D$ is a maximal $\A$-order,  
it must have the form $D=\mathrm{End}_{\A}(L)$ for some full rank 
$\A$-lattice $L$. Since $N$ is the normalizer $N_{\mathrm{GL}_n(K)}(D)$,
we can write $a L = J L$, for some fractional $\A$-ideal $J$. Writing
$J_Q=b_Q\mathcal{O}_Q$ and taking the determinant, we obtain 
$\det (a) \det (L_Q)= b_Q^n \det (L_Q)$, for all $Q \neq P$. Therefore, $J^n$ is 
generated by the global element $\det a$.
The map $\phi$ sending $a$  to 
the ideal class $[J]$ of $J$ is a group homomorphism from
$N$ to the Picard group $\mathrm{Pic}(\A)$. Since 
$\phi(a)^n=\left[J^n\right]=[\det a]$, then $\phi(a)^n$ is 
the trivial class. Furthermore, for any ideal $I$, $D$
is also the endomorphism ring of $IL$. If $I^n$ is principal, 
then $IL$ is isomorphic to $L$ as an $\A$-module, whence there is a 
matrix $c$ normalizing $D$ that satisfies $cL=IL$. It follows that 
$\det c$ generates $I^n$. We conclude that the image of $\phi$ is
the annihilator of $n$ in the Picard group. On the other hand, 
if $\phi(a)$ is the trivial class, then $aL=IL$, where 
$I=\lambda \A$  is principal, and therefore the determinant
of $\lambda^{-1}a$ is a unit. We conclude that $\lambda^{-1}a$
belongs to $D^*$ locally at all places, and therefore also
globally.
\end{proof}

In the proof of Th. \ref{mainth-1} in \S \ref{ch4}, we need a finiteness result
that is proved in Prop. \ref{FBA1} below. Before proving it we need a
technical lemma.

\begin{lem}\label{lemma610}
Let $X$ be a smooth, projective, geometrically integral curve
defined over a finite field. Let $n\geq1$ be an integer. Then
there is a finite number of indecomposable vector bundles of rank $n$
or less over $X$, up to tensor product with $\LL^{rP}$ for $r\in\mathbb{Z}$.
\end{lem}

\begin{proof}
Let $g$ be the genus of $X$, and let $E$ be an indecomposable vector bundle
of rank $m\leq n$. Tensoring with $\mathcal{L}^{rP}$ changes the degree of $E$
by $rm\deg P$ and preserves indecomposability, so we may assume that 
$(g-1)m(m-1)<\deg E\leq (g-1)m(m-1)+m\deg P .$
This leaves finitely many possibilities for the pair $(m,\deg E)$. For each
of them, every indecomposable bundle of that rank and degree has all 
its Harder-Narasimhan slopes \cite{Harder-Narasimhan} non-negative 
\cite[Prop.~2.1]{Schiffmann}, and
there are only finitely many isomorphism classes of coherent sheaves of a
given rank and degree with this property \cite[\S 2.1, p.~307]{Schiffmann}.
\end{proof}

\begin{prop}\label{FBA1}
    The set of non-split maximal orders in the matrix algebra $\mathbb{M}_n(K)$
    corresponds to the vertices of a combinatorially finite complex whose 
    dimension is lower than $n-1$.
\end{prop}

\begin{proof}
    Choose one representative for each class of indecomposable vector bundles
    up to tensor product with an invertible bundle of the form $\LL^{rP}$, 
    including the one dimensional bundles. Consider all sequences of the form 
    $\Lambda_1,\dots,\Lambda_m$ of these representatives having dimensions that
    sum precisely $n$, but with $m<n$. The bundles of the form
    $\LL^{r_1P}\otimes_{\oo_X}\Lambda_1\oplus\cdots\oplus
    \LL^{r_mP}\otimes_{\oo_X}\Lambda_m$
    fall in a subset of an apartment whose realization is homeomorphic to
    an affine subspace of dimension $m-1<n-1$, and  there is only a finite 
    number of these sequences, according to Lemma \ref{lemma610}.
\end{proof}

%\begin{comment}

\section{GSAs, QSAs and good sectors}\label{split apartments}

Our strategy to compute the C-complex,
as defined in 
\S\ref{section main results},
consists in studying the images via
the canonical projection of the split apartments.
According to Th. \ref{mainth-1} and Th. \ref{mainth1},
these sets cover the C-complex, except for a ``combinatorially
small set'', i.e., an $r$-neighborhood of a combinatorially 
finite $(n-2)$-dimensional complex, as defined 
in \S \ref{section main results}. As Th. \ref{mainth1} shows, 
computing these images reduces to studying the action of a suitable 
Weyl group. In this section, we introduce some tools that are needed 
to prove both theorems. Their proofs are completed in \S\ref{ch4}.

\begin{Def}
Let $\hat{\Sigma}_{\BB,P,\beta}$ be the image in $\mathscr{C}_P(D)$ of the GSA $\Sigma_{\BB,P,\beta}$. 
This is an element in the category $\mathfrak{sing}$ defined in \S\ref{SSsing}. It is called a quotient split apartment, or a QSA.
\end{Def}

Here, we keep the notations  on divisors from 
\S \ref{section conventions}.
We let the (Weyl) group $S_n$ act by permutations on the abelian group 
$\mathrm{Div}(X)^{n}$ of divisor vectors. This is denoted in the 
sequel as follows:
\begin{equation}\label{eq action of weyl group}
\sigma \cdot \BB= \sigma\cdot\big(B_1,\dots,B_n \big)=
\big(B_{\sigma^{-1}(1)},\dots,B_{\sigma^{-1}(n)}\big),
\end{equation}
whenever $\BB=(B_1, \cdots, B_n) \in \mathrm{Div}(X)^{n}$ and 
$\sigma \in S_n$.
This action of $S_n$ on $\mathrm{Div}(X)^{n}$ induces an action of $S_n$
on $\mathrm{Div}(\A)^{n}$ by essentially ignoring the $P$-coordinate in 
each divisor. This action also descends to the quotient
$\mathrm{Pic}(\A)^n$, where $\mathrm{Pic}(\A)$ is the ideal class
group of $\A$. Since $S_n$ acts trivially on 
$\Delta_{\mathrm{Pic}(\A)} \subset \mathrm{Pic}(\A)^{n}$, 
the preceding action factors through
$\nabla_{\mathrm{Pic}(\A)}=\mathrm{Pic}(\A)^{n}/
\Delta_{\mathrm{Pic}(\A)}$.
In the sequel, we respectively denote by $(\BB)$ and $\BBB$ the classes 
of $\BB \in \mathrm{Div}(X)^{n}$ in $\mathrm{Pic}(\A)^{n}$ and 
$\nabla_{\mathrm{Pic}(\A)}$.

\begin{lem}\label{lemma action of H}
Let $H=H_{\BBB}\subseteq S_n$ be the stabilizer
of the class $\BBB\in\nabla_{\mathrm{Pic}(\A)}$. Then $H$ acts on the 
split apartment $\Sigma_{\BB,P,\beta}$ via a subgroup
$H'$ of the normalizer $N=N_{\mathrm{GL}_n(K)}(D)$ of 
$D=\DD_{\BB,\beta}(U_P)$. The group $H'$
contains the group $\mathbb{D}(k)$ of linear maps whose matrix representation, in the base $\beta$, consists of invertible diagonal
matrices with entries in $k^{*}$, and $H\cong H'/\mathbb{D}(k)$.
\end{lem}

\begin{proof}
Let $G'\subseteq\mathrm{GL}_n(K)$ be the group generated by
linear transformations whose matrix expression in the base 
$\beta$ are either invertible diagonal matrices or permutation matrices.
Note that this group is a semidirect product 
$\mathbb{D}\rtimes S_n$, 
where $\mathbb{D}$ is the group of diagonal matrices.
Let $H'$ be the $*$-stabilizer of $\Sigma_{\BB,P,\beta}$ in $G'$. 
Then $H'$ acts on $\Sigma_{\BB,P,\beta}$ via the induced $*$-
action, and $H'\subseteq N$, since $D$ is the 
restriction to $U_P$ of every maximal order corresponding to 
a vertex of $\Sigma_{\BB,P,\beta}$.
It is immediate that $\mathbb{D}(k)$ is contained in $H'$ and, 
since $\nu(k^{*})=\lbrace 0\rbrace$, it acts 
trivially on the apartment $\Sigma_{\BB,P,\beta}$.
Furthermore, we claim that 
$\mathbb{D}(k)=\mathbb{D}\cap N$, which proves that 
$\mathbb{D}(k)=\mathbb{D}\cap H'$, and therefore the quotient
$H_1\cong H'/\mathbb{D}(k)$ can be naturally identified
with a subgroup of $S_n$ that acts on $\Sigma_{\BB,P,\beta}$. 
It only remains, therefore, to prove the claim,
and the identity $H_1=H$.

Next we prove the claim. For this, we just need to prove 
that $\mathbb{D}\cap N\subseteq \mathbb{D}(k)$.
Let $d$ be the matrix whose expression with respect to the basis 
$\beta$ is the diagonal matrix $\mathrm{diag}(f_1,\dots,f_n)$. 
Recall that we have the formula
\begin{equation}\label{eq action of d}
d\DD_{\BB,\beta}d^{-1}=\DD_{\BB+\mathrm{div}(\vec F),\beta},\quad
\text{where}\quad\mathrm{div}(\vec F)=
\big(\mathrm{div}(f_1),\dots,\mathrm{div}(f_n)\big).
\end{equation}
Assume some diagonal element $f_i$ is non-constant, i.e., 
assume that $d \notin \mathbb{D}(k)$. Then its divisor
has a non-trivial coefficient at some place $Q$ distinct 
from $P$. It follows that
$d$ cannot normalize the completion of $D$ at $Q$. 
Hence $d \notin \mathbb{D}\cap N$, and therefore the claim follows.

Now take $\sigma\in H_1$. This means that there are matrices $p$ and $d$
such that, in the basis $\beta$, $d$ is diagonal as in 
Eq. \eqref{eq action of d}, $p$ is the permutation
matrix corresponding to $\sigma$, and $dp$ normalizes the apartment
$\Sigma_{\BB,P,\beta}$. This implies that 
$\DD_{\sigma\cdot\BB+\mathrm{div}(\vec F),\beta}=
\DD_{\BB+\vec vP,\beta}$, for some $\vec v\in\ZZ^n$,
and therefore
$$\sigma\cdot\BB+\mathrm{div}(\vec F)=\BB+\vec vP+\vec\delta B_0,$$
for some divisor $B_0$. By taking classes in $\nabla_{\mathrm{Pic}(\A)}$,
we conclude that $\sigma$ stabilizes the class $\BBB$. This
proves that $H_1\subseteq H$. The statement $H\subseteq H_1$ is proved
analogously since the preceding relation allows us to construct an element
$dp$, as above, that normalizes the apartment. The result follows.
\end{proof}

Next result follows directly from Eq. \eqref{eq action of d},
since every element in $G'$ has the form $dp$, where $d$ is
diagonal on the basis $\beta$, and $p$ is a permutation matrix
in the same basis.

\begin{Corollary}
The group $H$ defined in Lemma \ref{lemma action of H} acts on 
the vertex set of the split apartment $\Sigma_{\BB,P,\beta}$ via
\begin{equation}\label{EqHaction}
\sigma \bullet \DD_{\BB,\beta}=h_\sigma * \DD_{\BB,\beta}=\DD_{\sigma\cdot\BB+
\mathrm{div}(\vec F_\sigma),\beta},
\end{equation}
where $\sigma \in H$ and $\vec F_\sigma$ is the diagonal 
component of any element $h_\sigma\in H'$ corresponding to 
$\sigma$. \qed
\end{Corollary}

Next result is immediate.

\begin{Corollary}\label{cor cont}
There exists a simplicial map $f_{\BB,\beta}$ from the quotient
$H\backslash\Sigma_{\BB,P,\beta}$ to $\hat{\Sigma}_{\BB,P,\beta}$.\qed
\end{Corollary}

It is proved in next section that the map $f_{\BB,\beta}$ defined above is an isomorphism.
For the rest of this section we need to fix an isomorphism between the GSA 
$\Sigma_{\BB,P,\beta}$ and a singular complex $\mathscr{R}$ 
whose subjacent topological space is $\nabla_{\mathbb{R}}=
\mathbb{R}^n/\langle(1,\dots,1)\rangle\cong\mathbb{R}^{n-1}$,
whose elements are denoted as classes  $[x_1,\dots,x_n]$ with the convention that
 $[x_1,\dots,x_n]=[x_1+c,\dots,x_n+c]$. We assume that the class of the vertex 
 corresponding to the 
maximal order $\DD_{\BB',\beta}$, with $\BB'=(B'_1,\dots,B'_n)$ is identified 
with the class $[r_1,\dots,r_n]$ of $(r_1,\dots,r_n)$, where $r_i=\deg(B'_i)$,
and we assume that a simplex in $\nabla_{\mathbb{R}}$ is the convex hull of its
vertices. Note that the origin $x_0=[0,\dots,0]$ need not be a vertex in the 
complex $\mathscr{R}$, and note that two such complexes corresponding to 
different QSA are only non-canonically isomorphic.

\begin{Def}\label{d73}
The geometrical sector
corresponding to the permutation $\sigma$ is the subset $\mathbf{Q}_\sigma
\subseteq\nabla_{\mathbb{R}}$ defined by
$$\mathbf{Q}_\sigma=\left\{[x_1,\dots,x_n]\big|x_{\sigma(1)}<\cdots<x_{\sigma(n)}\right\},
\qquad x_1,\dots,x_n\in\mathbb{R}.$$
A fine sector corresponding to the permutation $\sigma$ is a subcomplex
$\mathscr{Q}(\sigma,v)$ of $\Sigma_{\BB,P,\beta}$ whose realization corresponds
to a subset of the form $v+\overline{\mathbf{Q}_\sigma}$, where 
$v\in\mathbf{Q}_\sigma$ is a vertex. Note that the set 
$v+\overline{\mathbf{Q}_\sigma}$
is fully contained in the sector $\mathbf{Q}_\sigma$. 
We call the sector $\mathscr{Q}(\sigma,v)$ good whenever 
the class $[r_1,\dots,r_n]$ corresponding
to $v$ satisfies $r_{\sigma(i+1)}-r_{\sigma(i)}>\kappa$ for 
any constant $\kappa$ satisfying the Riemann-Roch relation
$$\dim_k\left(\LL^{B'}(X)/\LL^{B'-P}(X)\right)=[k(P):k],$$
for every divisor $B'$ whose degree exceeds $\kappa$ 
(cf.~\cite[\S 1, Th.~1.5.17]{Stichtenoth}). Note that whenever $w$ 
is a vertex in the good sector $\mathscr{Q}(\sigma,v)$, the subsector 
$\mathscr{Q}(\sigma,w)$  is also good. When $\mathscr{Q}(\sigma,v)$
is a fine or good sector with realization $v+
\overline{\mathbf{Q}_\sigma}$, the open sector 
$\mathbf{Q}(\sigma,v)=v+\mathbf{Q}_\sigma$ is also called fine or good.
\end{Def}

\begin{lem}\label{enoghcpa}
Let $\DD$ be the maximal order corresponding to a vertex $v$ in a good sector. Then the image $\mathbb{D}$ of the ring of global sections $\DD(X)$ in the local residual ring $\DD_P/\pi_P\DD_P$ has the form
\begin{equation}\label{formofD}
\mathbb{D} \cong \begin{pmatrix}
 %& & &\\
k & k(P) &\cdots  & k(P) \\
 0 & k &  &\vdots \\
 \vdots & & \ddots & k(P) \\
 0 &\cdots & 0  & k
 \end{pmatrix}
 \subseteq\mathbb{M}_n\big(k(P)\big),
\end{equation}
 in the basis $\beta$ corresponding to any GSA 
 $\Sigma_{\BB,P,\beta}$ containing the sector.
\end{lem}

\begin{proof}
Permuting the basis or decreasing the sector if needed, we can assume 
that the good sector is $\mathscr{Q}=\mathscr{Q}(\sigma_0,v)$, where
$\sigma_0\in S_n$ reverses the order of $\{1,\dots,n\}$. Then 
$\DD=\DD_{\BB',\beta}$ with $\BB'=(B'_1,\dots,B'_n)$, where 
$\deg(B'_{i}-B'_{i+1})>\kappa$. Recall that this order is, by definition
$$\DD = \begin{pmatrix}
\oo_X & \LL^{B'_1-B'_2}&\cdots  & \LL^{B'_1-B'_n} \\
 \LL^{B'_2-B'_1} & \oo_X &  &\vdots \\
 \vdots & & \ddots & \LL^{B'_{n-1}-B'_n} \\
 \LL^{B'_n-B'_1} &\cdots & \LL^{B'_n-B'_{n-1}}  & \oo_X 
 \end{pmatrix}.$$
 Now we recall that $\LL^{B''}(X)=0$ when 
 $\deg{B''}<0$, while $\oo_X(X)=k$.
 This takes care of all coefficients on or below the main diagonal. 
 For the coefficients above the diagonal we note that an element 
 in $\LL^{B''}(X)$ is trivial in $\LL^{B''}_P/\pi_P\LL^{B''}_P$ 
 precisely when it belongs to $\LL^{B''-P}(X)$. The result follows
 now from the definition of $\kappa$.
\end{proof}

\begin{lem}\label{enoghcp}
Let $\mathfrak{D}$, $\mathbb{D}$ and $\Sigma_{\BB,P,\beta}$
as in Lemma \ref{enoghcpa}.
The natural action of the group $\mathbb{D}^*$ on the star of the vertex $v$ 
corresponding to $\mathfrak{D}$ has precisely one orbit for each simplex of
the star that is contained in $\Sigma_{\BB,P,\beta}$. 
Furthermore, there is a unique chamber that is invariant with 
respect to this action, and the stabilizer 
$\mathrm{Stab}_N(v)$ in $N$ acts with precisely the same orbits.
\end{lem}

\begin{proof}
The simplices having $v$ as a vertex are in correspondence with the 
flags in the residual space $k(P)^n$. 
The flags corresponding to simplices in the apartment are those whose 
constituent spaces are generated by subsets of the corresponding basis. 
It suffices, therefore, to show that $\mathbb{D}^*$ acts on the flags 
having one separate orbit for each flag of the form $\{0\}=W_0\subset 
W_1\subset \cdots\subset W_s=V$, where each $W_i$ is generated by a 
subset of the basis $\bar\beta=\{\bar{\mathbf{v}}_1,\dots,
\bar{\mathbf{v}}_n\}$, and no additional orbits. 
Observe first that the flag corresponding to 
$\{0\}=V_0\subset V_1\subset \cdots\subset V_n=V$, with $V_i$ is 
generated by $\{\bar{\mathbf{v}}_1,\dots, \bar{\mathbf{v}}_i\}$, 
is an invariant flag, corresponding to an invariant
chamber.

Next we show that no two simplices in the apartment can belong to the same orbit. Consider two flags $\{0\}=W_0\subset W_1\subset \cdots\subset 
    W_s=V$ and $\{0\}=W'_0\subset W'_1\subset \cdots\subset W'_s=V$, 
    with each 
    $W_i$ or $W'_i$ generated by a subset of $\bar{\beta}$. 
    Choose $r$ so that
    $W_r\neq W'_r$, but $W_i= W'_i$ for $i<r$. 
    In particular $W_{r-1}= W'_{r-1}$,
    so they are generated by the same subset 
    $\tau$ of $\bar{\beta}$. On the 
    other hand, $W_r$ and $W'_r$ are generated by $\tau\cup
    \{\bar{\mathbf{v}}_i\}$ and $\tau\cup\{\bar{\mathbf{v}}_{i'}\}$, 
    respectively. Assume $i<i'$ to fix ideas. Then 
    $\dim(W_r \cap V_i) > \dim(W_{r-1} \cap V_i)$, while 
    $\dim(W'_r \cap V_i) = \dim(W'_{r-1} \cap V_i)$.
    This precludes the two
    flags from being in 
    the same orbit. Note also that this proof works 
    for every group preserving the
    invariant chamber, like the 
    stabilizer of the vertex in $N$.
    This proves the last statement.

    Lastly,  we show that every orbit has a simplex in the apartment. 
    If a chamber is in the apartment so are all its faces, 
    so it suffices to assume $s=n$.  If $W_{n-1}=V_{n-1}$, an inductive
    reasoning applies. Assume now that $W_{n-1}\neq V_{n-1}$. Choose 
    $r\leq n-1$ so that $V_{r-1}\subseteq W_{n-1}$, but $V_r\not\subset 
    W_{n-1}$. Note that $\mathbb{D}^*$ acts on $V/V_r$. Now consider the 
    canonical isomorphism $$V/V_r=(W_{n-1}+V_r)/V_r\cong W_{n-1}/
    (W_{n-1}\cap V_r) = W_{n-1}/V_{r-1}.$$ Since 
    $\{\bar{\mathbf{v}}_{r+1},\dots,\bar{\mathbf{v}}_n\}$ is a basis for
    the space on the left, the space $W_{n-1}$ must have a basis
    of the form $\{\bar{\mathbf{v}}_1,\dots,\bar{\mathbf{v}}_{r-1},
    \bar{\mathbf{v}}_{r+1}+b_{r+1}\bar{\mathbf{v}}_r,\dots,
    \bar{\mathbf{v}}_n+b_n\bar{\mathbf{v}}_r\}$. It is easy to find an element
    of $\mathbb{D}^*$ sending the latter basis to $\bar\beta-
    \{\bar{\mathbf{v}}_r\}$, so we can assume that $W_{n-1}$ is the span of 
    that set.  Then the stabilizer $\mathbb{S}$ of $W_{n-1}$ in 
    $\mathbb{D}^*$ consists of matrices having $0$s in the $r$-th row. We 
    can permute the basis by placing $\bar{\mathbf{v}}_r$ at the end. Now
    $\mathbb{S}$ is a group of upper triangular matrices with an expression
    similar to (\ref{formofD}), except for some zeroes in the last column.
    Now the inductive step goes as before.
\end{proof}

\begin{Remark}
A result analogous to Lemma~\ref{enoghcp} is proved in \cite[Cor. 7.14]{BravoLoisel} for arithmetic groups $G=\mathtt{G}(R)$, where $\mathtt{G}$ is an arbitrary split group. This includes, in particular, the case $\mathtt{G}=\mathrm{SL}_n$, but not 
the case where $G=N$ is the normalizer of a maximal order.
\end{Remark}

\begin{prop}\label{smalllo}
    The complement, in the realization $\mathbf{\Sigma}$ of 
    $\Sigma_{\BB,P,\beta}$,
    of a union $\bigcup_{\sigma\in S_n}
    \mathbf{Q}(\sigma,v)$ of good open sectors is contained in a 
    $\kappa'$-neighborhood, for some constant $\kappa'$, of an 
    $(n-2)$-dimensional combinatorially finite complex.
\end{prop}

\begin{proof}
It suffices to note that, when a point is in the complement of that
    union, then it corresponds to a point in $\nabla_{\mathbb{R}}$ having
    two coordinates differing by less than $\kappa$, and therefore 
    it is at a bounded 
    distance from the affine subspace of points 
    for which those two coordinates are 
    equal. This subspace need not be a subcomplex, 
    but we claim that a shift is. To be
    precise, we can assume without loss of generality 
    that the two coordinates in 
    question are the first two coordinates. Consider any 
    order $\DD_{\BB,\beta}$
    corresponding to a vertex in the apartment. 
    Then we claim that the affine subspace
    $$\mathbf{L}=\left\{[x_1,\dots,x_n] \big| x_2-x_1=
    \deg(B_2-B_1)\right\},$$
    is a shift that is the realization of a subcomplex. 
    This is so since the points
    of the form $[\deg B_1,\deg B_2,\deg(B_3+m_3P),\dots,
    \deg(B_n+m_nP)]$ are
    the vertices of a simplicial complex whose realization 
    is $\mathbf{L}$, and every simplex
    of this complex corresponds to a simplex of $\Sigma_{\BB,P,\beta}$.
\end{proof}

\section{The ring of global sections of maximal orders}\label{ch4}

In this section we focus on the
algebraic aspects of the ring $\DD(X)$ and the action
of its group of units $\DD(X)^*$ on the neighbors 
of the corresponding vertex $v_\DD$ in
$\mathcal{B}_P(D)$, in order to describe its link in the quotient.
With this in mind, we look for a basis where the ring $\DD(X)$
of global sections has an easily described canonical form. 
By using that, we finally prove Ths. \ref{mainth-1}, \ref{mainth1},
\ref{mainth3} and \ref{mainth5}. In all
that follows, we assume that $\DD=\DD_{\vec{B},\beta}$, for some vector 
divisor $\vec B$ and some basis $\beta$. Since the $\A$-order $D$
is always split on some basis, any connected component of the
full C-complex has one of these vertices. 
Since the symmetric group acts naturally on the
set of vector divisors $\vec B$ it is always possible to assume a suitable
order on the basic vectors that makes $\DD(X)$ simpler. This is achieved 
inducing a partial order in the Picard group.

\begin{Def} Let $B_1,B_2$ be divisors on $X$. We  define a pre-order 
on $\mathrm{Div}(X)$ by   
$$B_1\preceq B_2 \,\,\,\, \Longleftrightarrow\,\,\,\, B_1\leq B_2 + 
\mathrm{div}(f), \,\,\, \text{ for some }f \in K^{*}.$$
If we have simultaneously both $B_1\preceq B_2$ and $B_2\preceq B_1$, then
$B_1\leq B_2+\mathrm{div}(f)$ and $B_2\leq B_1+\mathrm{div}(g)$, 
so in particular
$B_1\leq B_1+\mathrm{div}(fg)$, and $\mathrm{div}(fg)\geq0$, whence $fg$
is a constant. We conclude that $B_2\sim B_1$. This means that the relation 
$\preceq$ induces a partial order in $\mathrm{Pic}(X)$.
The class in $\mathrm{Pic}(X)$ of a divisor $B$ is denoted $\bar B$.
When referring to classes we write $\bar B_1\preceq \bar B_2$ instead of 
$B_1\preceq B_2$.
\end{Def}

 Let $B_i$ and $B_j$ be divisors on $X$ and suppose that there exists 
 $g\neq0$  in $\LL^{B_i- B_j}(X)$. This is equivalent to 
 $0\leq\mathrm{div}(g) + B_i - B_j$, i.e., $ B_j \leq  B_i + \mathrm{div}(g)$,
 or $\bar{B}_j \preceq \bar{B}_i$.
 We conclude that  $\LL^{B_i- B_j}
 (X) \neq \{0\}$ if and only if $\bar{B}_j \preceq \bar{B}_i$.
 It follows that the basis vectors can be ordered in a way
 that an element in $\DD(X)$ can have nontrivial coordinates only on 
 or above a suitable set of diagonal blocks. To be precise, 
 if the vector $\vec B$ has the form $$\vec B=
 \big(B^1_1, B^1_2, \dots,B^1_{n_1},B_{1}^{2},B_{2}^2,\dots, 
B_{n_2}^2,\dots,B_{1}^r, B_{2}^r,\dots, B_{n_r}^r\big),$$
where the following conditions are satisfied:
\begin{itemize}
\item $n=n_1+\dots+n_r$, 
\item $B_i^j$ is linearly equivalent to 
$B_{i'}^{j'}$ if and only if $j=j'$, and 
\item the inequality $B_i^j\preceq B_{i'}^{j'}$ implies $j\geq j'$,
\end{itemize}
then the ring of global sections 
$\DD_{\vec{B},\beta}(X)$ has, in the basis $\beta$, the form
\begin{equation}
\Big(\LL^{B_i-B_j}(X)\Big)_{i,j=1}^{n}\cong\begin{pmatrix}
 %& & &\\
\delta_1\mathbb{M}_{n_1}(k)\delta_1^{-1} & \ast &\cdots  & \ast \\
 0 & \delta_2\mathbb{M}_{n_2}(k)\delta_2^{-1} &  &\vdots \\
 \vdots & & \ddots & \ast \\ \\
 0 &\cdots & 0 & \delta_r\mathbb{M}_{n_r}(k)\delta_r^{-1}
 %\\ & & & 
 \end{pmatrix} \label{block_order},
\end{equation}
where $\delta_d$ is a diagonal matrix $\delta_d=\mathrm{diag}(f_{d,1}, \dots, f_{d,n_d})$ with $\mathrm{div}(f_{d,i})=B_{i}^{d}-B_{1}^d$.
When this is the case we say that the vector divisor $\vec B$ is 
conveniently ordered, or that satisfies condition \textbf{CO}.

The preceding expression yields an analogue of the Levi decomposition
of a parabolic subgroup of $\mathrm{GL}_n$. We refer to the diagonal blocks 
in the preceding decomposition as the Levi blocks of the order 
$\DD_{\vec B,\beta}$, and adopt other notations that exploit the 
similarity between this decomposition and suitable decompositions 
of parabolic subgroups of algebraic groups. This depends 
in principle on the basis $\beta$, since we might have a non-trivial identity
of the form $\DD_{\vec B,\beta}=\DD_{\vec B',\beta'}$, but the results below
show that this is not the case.
%To further simplify the subsequent argument, we assume \orange{
%until Cor. \ref{c851}}  that $\beta=\epsilon$ is the canonical basis,
%and write $\DD_{\vec{B}}$ instead of $\DD_{\vec{B},\epsilon}$.
The semisimple algebra associated to the order $\DD_{\vec{B},\beta}$
is, by definition, the algebra
\begin{equation}\label{eq ss part}
    \mathbb{D} =\DD_{\vec{B},\beta}(X)/\mathcal{R}\cong 
    \bigoplus_{d=1}^r\mathbb{M}_{n_d}(k),
\end{equation}
where $\mathcal{R}$ is the nilradical of the algebra $\DD_{\vec{B},\beta}(X)$. 
In particular, the sizes $n_i$ of the blocks depend only on the conjugacy 
class of the order. 

By a generalized permutation matrix, we mean a matrix permuting the factors 
of the decomposition $K^n=\bigoplus_{i=1}^{n}K\vec{e}_i$. The sub-sum 
$K^n_s=\bigoplus_{i=n_1+\dots+n_{s-1}+1}^{n_1+\dots+n_s}K\vec{e}_i$
is called the $s$-th block subspace associated to $\vec B$.
Note that each Levi block above can be seen as a linear map
on the corresponding block subspace.
The decomposition $K^n=\bigoplus_{s=1}^rK_s^n$ is called the
block decomposition associated to $\vec B$. By definition,
the Levi subgroup associated to $\DD_{\vec{B},\beta}$ is the group of matrices
preserving this decomposition. The associated parabolic subgroup
is the set of matrices preserving each sub-sum $\oplus_{s=1}^{s'} K_s^n$, 
$1 \leq s' \leq r$ (compare to \cite[\S 3.3]{AlgGrGLn}).

Let $A, B$ be finite-dimensional $k$-algebras, $J(A)$, $J(B)$ their 
Jacobson radicals, and let $f: A\to B$ be a surjective algebra 
homomorphism whose kernel contains no nonzero idempotents. 
Next statements hold from Idempotent Lifting Theorem 
(See \cite[Th. 6.18]{Li}):

 \begin{enumerate}
 \item[(i)] The homomorphism $f$ maps $J(A)$ onto $J(B)$ and $A^*$ onto $B^*$.
\item[(ii)] For any primitive idempotent $i$ in $A$, $f(i)$ is a primitive 
idempotent in $B$.
\item[(iii)] For any idempotent $j$ in $B$ there is an idempotent $i$ 
in $A$ such that $f(i) = j$. Furthermore, if $j$ is primitive, then 
any such $i$ is also primitive.
\item[(iv)] Any two primitive idempotents $i, i'$ in $A$ are conjugate in
$A$ if and only if $f(i), f(i')$ are conjugate in $B$.
\end{enumerate}

\begin{lem}\label{ncor1}
In the preceding notation,
if $1=P_1+\cdots+P_n$ is a decomposition of $1$ as a sum of 
orthogonal idempotents in $B$, then there are orthogonal liftings $Q_i$ 
of $P_i$ to $A$ such that $1=Q_1+\cdots+Q_n$.
\end{lem}

\begin{proof}
It can be assumed that each $P_i$ is primitive or we can refine the 
decomposition. The existence of a lifting $Q_1$ of $P_1$ to $A$ follows 
from statement (iii) above. Now $P_2,\dots,P_n$ are in the image
$B'\subseteq B$ of $A'=(1-Q_1)A(1-Q_1)$. The result follows by applying 
the inductive hypotheses to the restriction of $f':A'\rightarrow B'$
of $f$. Note that the unity $1-P_1$ of $B'$ equals the sum $P_2+\cdots+P_n$.
\end{proof}

\begin{lem}\label{ncor0}
In the preceding notation,
any two idempotents in $A$ are conjugate if and only if their images in $B$
are conjugate. In particular, two idempotents with the same image are 
conjugate.
\end{lem}

\begin{proof}
Any idempotent can be written as a sum of orthogonal primitive
idempotents (cf. \cite[Th. 6.2]{Li}). Assume the images of the 
idempotents $P$
and $P'$ are conjugate in $B$, and write $P=P_1+\cdots+P_n$ as a 
sum of primitive idempotents. Replacing $P'$ by a conjugate we can assume that $f(P')=f(P)=f(P_1)+\cdots+f(P_n)$. Then $f(P_i)$ is in the image $B'=f(P)Bf(P)\subseteq B$ of $P'AP'$. By the preceding lemma
we can write $P'$ as a sum of idempotents $P'=P'_1+\cdots+P'_n$
with $f(P_i')=f(P_i)$. Since $P_1$ and $P_1'$ are conjugate,
by statement (iv) above,
we can apply a conjugation and assume they are equal. Then
both $P_2,\dots,P_n$ and $P'_2,\dots,P'_n$ are in the algebra
$(1-P_1)A(1-P_1)$. Induction applies.
\end{proof}

The same reasoning can be used to prove the following result:

\begin{lem}\label{ncor}
In the notations of the previous result, if $1=P_1+\cdots+P_n$ and $1=Q_1+\cdots+Q_n$ are two decompositions of the identity into orthogonal idempotents in $A$ that satisfy $f(P_i)=f(Q_i)$, for $i=1,\dots,n$, then there exists an element $E\in A^*$ satisfying $EQ_iE^{-1}=P_i$ for $i=1,\dots,n$.\qed
\end{lem}

\begin{prop}\label{prop_canonical}
Let $\DD=\DD_{\vec{B},\epsilon}$ and $\DD'=\DD_{\vec{B'},\epsilon}$ be 
two maximal $X$-orders that the canonical basis $\epsilon$ splits. 
Assume that $\BB$ satisfies condition \textbf{CO}, as defined above.
Any matrix $C\in \mathrm{GL}_n(K)$ satisfying $\DD' = C\DD C^{-1}$  
has the form $C= PU$, where $P$ is a permutation matrix and 
\begin{equation}\label{equnipotent}
U = \begin{pmatrix}
 %& & &\\
C_1 & * &\cdots  & * \\
 0 & C_2 &  &\vdots \\
 \vdots & & \ddots & * \\
 0 &\cdots & 0  & C_r%\\ & & &
 \end{pmatrix}, 
\end{equation}
is a matrix in the parabolic subgroup associated to $\DD$.
Moreover, $C_d\in c_d\eta_d\mathbb{M}_{n_d}(k)\delta_d^{-1}$, where
$c_d\in K^*$, $\eta_d$ a diagonal matrix depending on $\BB$ and $\BB'$,
and $\delta_d$ is the diagonal matrix appearing in the corresponding Levi
block in \eqref{block_order}. Furthermore, $\vec{B}'$ is coordinate-wise 
linearly equivalent to $\sigma \cdot\vec{B}+B_0\vec \delta$, for some 
divisor $B_0$, where $\sigma$ is the permutation associated to $P$.
 \end{prop}

\begin{proof}
Denote by $\mathcal{R}$ and $\mathcal{R}'$ the Jacobson radicals of the global section rings $\DD(X)$ and $\DD'(X)$ respectively.
Applying a permutation if needed we can assume that $\vec B'$ also satisfies condition \textbf{CO}.
By Eq. \eqref{eq ss part}, the associated semisimple algebras are: 
\begin{equation}\label{dec2}
\mathbb{D}=\DD(X)/\mathcal{R}\cong  \displaystyle 
\bigoplus_{d=1}^r\mathbb{M}_{n_d}(k)\hspace{.5cm} \text{and} \hspace{.5cm} 
\mathbb{D}'= \DD'(X)/\mathcal{R}'\cong \displaystyle 
\bigoplus_{d=1}^{r'}\mathbb{M}_{m_d}(k).
\end{equation}
If there exists a matrix $C\in \mathrm{GL}_n(K)$ such that 
$\DD'=C\DD C^{-1}$, then the semisimple
parts $\mathbb{D}$ and $\mathbb{D}'$ are isomorphic, and,
by the Artin-Wedderburn Theorem \cite[Ch. XVII]{Lang}, the 
two preceding decomposition coincide up to permutation,  
i.e., $r=r'$, and there is a permutation $\sigma$
of $\{1,\dots,r\}$ such that $n_i=m_{\sigma(i)}$. The permutation
$\sigma$ is not unique, but we can make it precise by choosing it as 
follows: We let $P_1,\dots,P_r$ be the central idempotents in 
$\mathbb{D}$, defined from the decomposition in (\ref{dec2})
and we let $\tilde P_1,\dots,\tilde P_r$ be liftings to 
$\DD(X)$ satisfying $\tilde P_1+\dots+\tilde P_r=1$ 
(See Lemma \ref{ncor1}). Analogously, let $P'_1,\dots,P'_r\in 
\mathbb{D}'$ be the central idempotents defined from the second 
decomposition in (\ref{dec2}).
Then we assume that $C\tilde P_iC^{-1}$ is a lifting of $P'_{\sigma(i)}$.
If we write $S$ for any permutation matrix preserving contiguous blocks
and inducing the permutation $\sigma$ in the latter, and if we
replace $\DD'$ by $S^{-1}\DD'S$, we can assume
that $C\tilde P_iC^{-1}$ is a lifting of $P'_i$.  
In other words, we can assume that the permutation $\sigma$ is trivial 
and $\mathbb{D} = \mathbb{D}'$.
We assume this in the rest of the proof.
It is unclear a priori whether $\vec B'$ still satisfies
condition \textbf{CO} after the last adjustment, so this is proved 
in next paragraph.

Note that we can assume that the lifting $\tilde{P}_i$ is a diagonal 
 matrix, and in this case $\tilde{P}_i\in\DD'(X)\cap\DD(X)$. Furthermore, 
 both $\tilde{P}_i\in\DD(X)$ and $C^{-1}\tilde{P}_iC\in C^{-1}\DD'(X)C=
 \DD(X)$ have the same image in $\mathbb{D}$. 
It follows from Lemma \ref{ncor} that there is a matrix $E\in\DD(X)^*$ such
that $E\tilde{P}_iE^{-1}=C^{-1}\tilde{P}_iC$.
This means that $CE$ belongs to
the Levi subgroup associated to $\DD=\DD_{\vec B,\beta}$.
This implies that $C$ belongs to the corresponding parabolic subgroup, 
since so does $E\in\DD(X)^*$.
Furthermore, it follows, under the preceding assumptions, that
$\vec B'$ satisfies condition \textbf{CO}. In fact, if this is not the 
case, the ring of global sections $\DD'(X)$ has a matrix $\Delta$ with a 
nontrivial coefficient in a position that lies under the 
diagonal blocks from (\ref{block_order}). Then $C^{-1}\Delta C\in\DD(X)$ 
has the same property by the definition of the associated
parabolic subgroup. The claim follows.

Note that any matrix in a diagonal block of $C_d$ takes the 
corresponding  block of $\DD(X)$, which is of the form 
$\delta_d\mathbb{M}_{n_d}(k)\delta_d^{-1}$, to the corresponding block
of $\DD'(X)$, which has the form $\eta_d\mathbb{M}_{n_d}(k)\eta_d^{-1}$, 
for a diagonal matrix $\eta_d$ whose definition is analogous to that
of $\delta_d$, but replacing $\vec B$ by $\vec B'$.
In particular, conjugation by $\eta_d^{-1}C_d\delta_d$ defines
an automorphism of $\mathbb{M}_{n_d}(k)$, which is inner by Skolem-Noether Theorem \cite[Th. 2.7.2]{gille}. Since two matrices define the same conjugation if and only if
they coincide up to a scalar factor, this finishes the proof of the second statement of the proposition.

Now, we prove the last statement, i.e., that $\vec{B}'$ is 
coordinate-wise linearly equivalent with $\sigma \cdot\vec{B}$, where 
$\sigma$ is the permutation associated to $P$. As in the second 
paragraph of this proof, we are reduced to the case where $\sigma$ is trivial.
Write $\DD=\DD_{\BB,\beta} = \left(\LL^{B_{ij}}\right)_{i,j=1}^{n}$, 
where $B_{ij} := B_i - B_j$.
Let us write $C=(c_{ij})_{i,j=1}^{n}$ and $C^{-1}=(w_{ij})_{i,j=1}^n$. 
Recall that $\DD' = C\DD C^{-1}$.
Choose indexes $i,s,r,j$ such that $c_{ir}$ and $w_{sj}$
denote nonzero entries of $C_d$ and $C_{d'}^{-1}$ respectively.
Note that this assumes that all indexes are in the range
defined by the corresponding block.
Then, the lattice:
$$\LL^{B_{ij}'}E_{ij} = E_{ii}\DD_{\vec{B}',\beta}E_{jj} = E_{ii}
(C\DD_{\vec{B},\beta}C^{-1})E_{jj}= (E_{ii}C)\DD_{\vec{B},\beta}(C^{-1}E_{jj}),$$
contains the lattice:
$$\begin{array}{rcl}
(E_{ii}C)(\LL^{B_{rs}}E_{rs})(C^{-1}E_{jj})
 &=& \left(E_{ii}\left[\displaystyle 
\sum_{l,k=1}^{n}c_{lk}E_{lk}\right]\right)\left(\LL^{B_{rs}}E_{rs}\right)
\left(\displaystyle \left[\sum_{a,b=1}^{n}w_{ab}E_{ab}\right] E_{jj}\right),\\ \\
 &=& \LL^{B_{rs}}\left(\displaystyle 
 \sum_{k=1}^{n}c_{ik}E_{ik}\right)E_{rs}E_{ss}\left(\displaystyle 
 \sum_{a=1}^{n}w_{aj}E_{aj}\right),\\ \\
 &=& \LL^{B_{rs}}\left(\displaystyle 
 \sum_{k=1}^{n}c_{ik}E_{ik}E_{rs}\right)\left(\displaystyle 
 \sum_{a=1}^{n}w_{aj}E_{ss}E_{aj}\right),\\ 
 &=& c_{ir}w_{sj}\LL^{B_{rs}}E_{is}E_{sj}, \\ 
 &=& c_{ir}w_{sj}\LL^{B_{rs}}E_{ij}.
 \end{array}$$
 Therefore $\LL^{B_{ij}'}\supseteq \LL^{B_{rs}-\mathrm{div}(c_{ir}w_{sj})}$, which is 
 equivalent to $B_{ij}' \geq  B_{rs}-\mathrm{div}(c_{ir}w_{sj})$. This implies 
 $\bar{B}_{rs} \preceq  \bar{B}_{ij}'$ in $\mathrm{Pic}(X)$. 
 Since the indexes $i$ and $r$
 correspond to the same block, and so do $j$ and $s$, the divisor
 $\bar{B}_{rs}$ is linearly equivalent to $\bar{B}_{ij}$. We conclude 
 that $\bar{B}_{ij} \preceq  \bar{B}_{ij}'$. Now the reverse inequality
is proved by transposing the roles of $\vec B$ and $\vec B'$.
The result follows if we set $B_0=B'_1-B_1$.
\end{proof}

\begin{Corollary}\label{c851}
In the notations of Prop.~\ref{prop_canonical},
when $\BB$ fails to satisfy the condition \textbf{CO}, any
matrix $C\in \mathrm{GL}_n(K)$ satisfying 
$\DD' = C\DD C^{-1}$ has the form $C= PUP'$, where $P$ and 
$P'$ are two permutation matrices and $U$ is as defined in Eq. \eqref{equnipotent}.
Furthermore, $\vec{B}'$ is coordinate-wise linearly equivalent with 
$\sigma \cdot\vec{B}+B_0\vec \delta$, for some divisor $B_0$,
where $\sigma$
is the permutation associated to $PP'$.
\end{Corollary}

\begin{proof}
By applying to $\DD$ a conjugation 
by a permutation matrix $P'$ 
if needed, we are reduced to the case where $\BB$ does 
satisfy condition \textbf{CO}. Thus, the result follows 
directly from Prop.~\ref{prop_canonical}.
\end{proof}

\begin{Corollary}\label{cor N and H action are the same}
Let $\vec B\in \mathrm{Div}(X)^n$, 
let $\beta$ be a basis and let $P$ be a place of $X$
as before. Set $D=\DD_{\BB,\beta}(U_P)$.
No two  vertices in $\Sigma_{\BB,P,\beta}$ have the same image in the 
C-complex $\mathscr{C}_P(D)$, unless they are in the same $H$-orbit.
\end{Corollary}

\begin{proof}
    By applying Prop. \ref{prop_canonical}, two vertices 
    $\DD_{\BB+\vec aP,\beta}$ and $\DD_{\BB+\vec bP,\beta}$ 
    in $\Sigma_{\BB,P,\beta}$ have the same image in $\mathscr{C}_P(D)$
    if and only if $\BB+\vec bP$ is linearly equivalent to 
    $\sigma \cdot(\BB+\vec aP)+B_0\vec \delta$, for some permutation 
    $\sigma$ and some divisor $B_0$. By looking at the image in 
    $\nabla_{\mathrm{Pic}(\A)}$, we conclude that $\sigma$ belongs to the
    stabilizer $H$ of the class $\BBB$ of $\BB$.
    Thus $\DD_{\BB+\vec aP,\beta}$ and $\DD_{\BB+\vec bP,\beta}$ 
    belong to the same $H$-orbit.
\end{proof}

Next result is crucial to prove Th. \ref{mainth1}.

\begin{prop}\label{prop H N action on simplices}
In the notations of Cor. \ref{cor N and H action are the same}, 
if $\{v_1,\dots,v_m\}$ and $\{v'_1,\dots,v'_m\}$ are the respective sets
of vertices of two  simplices $\boldsymbol\sigma$ and $\boldsymbol\sigma'$
in  $\Sigma_{\BB,P,\beta}$, and if there is an element $g\in N$ satisfying 
$g\cdot v_i=v'_i$ for each $i=1,\dots,m$, then there is an element 
$h\in H$ satisfying $h\cdot v_i=v'_i$ for each $i=1,\dots,m$.
\end{prop}

\begin{proof}
    Let $\boldsymbol\sigma$ and $\boldsymbol\sigma'$ be two simplices 
    with the same image in $\mathscr{C}_P(D)$. Then there is an
    element $g\in N=\mathrm{Norm}_{\mathrm{GL}_n(K)}(D)$, where 
    $D=\DD_{\BB,\beta}(U_P)$, satisfying 
    $g\cdot\boldsymbol\sigma=\boldsymbol\sigma'$. Let $v$ be a vertex
    in $\boldsymbol\sigma$. Note that we can assume that $v$
    is the vertex corresponding to $\DD_{\BB,\beta}$ or we redefine
    $\BB$.  Then $g\cdot v$ is in the same $N$-orbit
    as $v$ and therefore also in the same $H$-orbit, by 
    Cor. \ref{cor N and H action are the same}. Say 
    $g\cdot v=h\cdot v$.
    Replacing $\boldsymbol\sigma'$ by  $h^{-1}\cdot\boldsymbol\sigma'$,
    and $g$ by $h^{-1}g$, we can assume that $\boldsymbol\sigma$ and 
    $\boldsymbol\sigma'$ have $v$ as a common vertex, and $g$ 
    stabilizes $v$. Since $g$ stabilizes $v$, it normalizes the order  $\DD_{\BB,\beta}$, so it acts on the associated residual space
    as a linear map $\bar{g}$.

     By setting $\vec B'=\vec B$ in Prop. \ref{prop_canonical}, we can 
     assume that $U=g$ has the form given in Eq. \eqref{equnipotent}. 
    In particular, by passing to 
    the corresponding residual space, $\bar{g}$ preserves the 
    associated (possibly non-complete) flag
    $$\{0\}=W_0\subseteq W_1\subseteq W_2\subseteq\cdots\subseteq 
    W_r=k(P)^n,$$
    of residual subspaces, with each $W_j$ spanned by a subset of 
    the basis $\bar{\beta}$ of the residual space induced from $\beta$. 
    On the other hand, the simplices 
    $\boldsymbol\sigma$ and $\boldsymbol\sigma'$ correspond to residual flags $$\{0\}= V_0\subseteq V_1\subseteq\cdots\subseteq  V_m=
    k(P)^n,\quad\textnormal{and}\quad\{0\}= V'_0\subseteq 
    V'_1\subseteq\cdots\subseteq V'_m= k(P)^n.$$
    The spaces $V_i$ and $V_i'$ are also generated by elements
    of $\bar{\beta}$, as the corresponding simplices are
    in the apartment. Note that $\bar{g}$ takes $V_i$ to
    $V'_i$, and therefore also $W_j\cap V_i$ to $W_j\cap V'_i$, for every
    $i\in\{0,\dots,m\}$ and every $j\in\{0,\dots,r\}$. In particular,
    the latter two spaces have the same dimension.

    Since all three flags consist of spaces generated by subsets of 
    $\bar{\beta}$, there is a permutation matrix mapping each
    $W_j\cap V_i$ to $W_j\cap V'_i$. In particular, it preserves the subspaces $W_j=V_m\cap W_j=V'_m\cap W_j$. This proves that it 
    belongs to $H$ since it belongs to the associated parabolic
    group, and therefore it fixes the partition described in the
    paragraph preceding Eq. \eqref{block_order}, which is the same as
    stabilizing $\BBB$. Analogously, it maps $V_i$ to $V'_i$, and therefore
    $\boldsymbol\sigma$ to $\boldsymbol\sigma'$. The result follows.
\end{proof}

\begin{proof}[Proof of Th. \ref{mainth1}]
It suffices to prove that the simplicial map 
$$f=f_{\BB,\beta}:H \backslash \Sigma_{\BB,P,\beta} \to \hat{\Sigma}_{\BB,P,\beta},$$
introduced in Cor. \ref{cor cont} is an isomorphism in the category
$\mathfrak{sing}$. First we prove that the map it induces in the category
of topological spaces is a homeomorphism. 
Let $\mathbf{A}$ be the underlying topological space of $\Sigma_{\BB,P,\beta}$,
and let $\mathbf{S}$ be the underlying topological space of 
$\hat{\Sigma}_{\BB,P,\beta}$. Denote by $\tilde{f}:\mathbf{A}
\rightarrow\mathbf{S}$ the induced continuous map. Surjectivity is 
immediate by definition.

First we prove injectivity. Assume $\tilde{f}(x)=\tilde{f}(x')$. Then $x'=g\cdot x$
for some $g\in N$. Since the action is simplicial, then $x$ and $x'$ 
are inside simplices $\boldsymbol\sigma$ and 
$\boldsymbol\sigma'=g\cdot\boldsymbol\sigma$, respectively.
Let $\{v_1,\dots,v_m\}$ be the vertices of $\boldsymbol\sigma$,
and let $\{v'_1,\dots,v'_m\}$ be the vertices of $\boldsymbol\sigma'$,
chosen in a way that $g\cdot v_i=v'_i$. Then $x=\sum_{i=1}^m\lambda_iv_i$
and $x'=\sum_{i=1}^m\lambda_iv'_i$ for some coefficients $\lambda_i$
satisfying $\sum_{i=1}^m\lambda_i=1$. Then, by Prop. 
\ref{prop H N action on simplices}, there is an element $h\in H$ 
satisfying $h\cdot v_i=v'_i$, whence $h\cdot x=x'$. This proves injectivity.

Recall that there is an invariant metric $\mu$
in the Bruhat-Tits building, see for instance Th.~11.6 in \cite[\S 11.2]{AbramenkoBrown}. This means that $N$ acts on the Bruhat-Tits 
building via isometries, and so does $H$ on the apartment.
Therefore the action of $N$ induces the metric
$\bar{\mu}(\bar{x}, \bar{x}') = \mathrm{min}\left\lbrace \mu(g\cdot x, 
g' \cdot x') \vert g,g'\in N \right\rbrace$, and
a similar property holds for the action of $H$.
Then the continuous map induced by $f_{\BB,\beta}$ is an isometry and the same 
holds for the inverse. In particular, the inverse is also continuous.

To end the proof that $f_{\BB,\beta}$ is an isomorphism in the category
$\mathfrak{sing}$, we need to see that the homeomorphism above induces a
bijection between the corresponding collections of singular simplices that
preserves permutations and faces, but this is straightforward from the preceding
argument.
\end{proof}

\begin{prop}\label{HNap35}
    If $v$ is a vertex in an open fine sector of $\Sigma_{\BB,P,\beta}$, then no
    two simplices in $\Sigma_{\BB,P,\beta}$ having $v$ as a vertex can be identified in the quotient $\mathscr{C}_P(D)$. In particular, if $v$ is 
    a vertex in an open good sector of $\Sigma_{\BB,P,\beta}$, then $v$ is unramified for the  cover $\phi:\mathscr{S}_P(D)\twoheadrightarrow
    \mathscr{C}_P(D)$. 
\end{prop}

\begin{proof}
    The first statement follows from Prop. \ref{prop H N action on simplices}
    and the fact that all simplices in $\Sigma_{\BB,P,\beta}$ having $v$ as a 
    vertex fall into a fundamental region for the $H$-action. The last 
    statement follows since, for any pre-image $w$ of $v$ in $\mathscr{S}_P(D)$, 
    any two simplices having $w$ as a vertex lie already in the image of 
    $\Sigma_{\BB,P,\beta}$ by Lemma \ref{enoghcp}.
\end{proof}

Now, we are able to prove Ths. \ref{mainth-1}, \ref{mainth3} and \ref{mainth5}.

\begin{proof}[Proof of Th. \ref{mainth-1}]
For each QSA $\widehat{\Sigma}_{\BB,P,\beta}$, we fix a lifting $\Sigma_{\BB,P,\beta}$
in $\mathscr{B}_P(D)$ that is a GSA. Let $\mathbf{A}$ be the subjacent topological
space of the GSA, which is homeomorphic to the $(n-1)$-dimensional affine space.
In fact, we can choose the homeomorphism $\mathbf{A}\cong\nabla_{\R}$ described in the
paragraph preceding Def. \ref{d73} for simplicity. Since the finite group $H$ acts 
on $\mathbf{A}$ by affine transformations, preserving the degree of a divisor,
it acts on $\nabla_{\R}$ via $\sigma^{-1}\cdot[x_1,\dots,x_n]=
[x_{\sigma(1)},\dots,x_{\sigma(n)}]$. This action can be extended to an action of
the whole of $S_n$ on $\mathbf{A}$. Now choose a full set of representatives 
$\{\sigma_1,\dots,\sigma_h\}$ for the cosets of $H$ in $S_n$, and then 
consider good sectors of the form $\mathscr{Q}(\sigma_i,v)$ inside 
$\Sigma_{\BB,P,\beta}$. Since each complex $\mathscr{Q}(\sigma_i,v)$ lies
inside a fundamental region for the extended $S_n$-action, the union is
inside a fundamental set (not a region, as its interior is not necessarily 
connected) for the $H$-action. In particular, these sectors embed into 
the quotient, and because of Lemma \ref{enoghcp},
they cannot meet the rest of the complex except at the boundary of the sector.
Our cusp set consists of these complexes $\mathscr{Q}(\sigma_i,v)$ for each QSA.
Now it suffices to see that the rest of the quotient lies in a tubular  
neighborhood of an $(n-2)$-dimensional combinatorially finite complex, but this 
follows from Prop. \ref{FBA1} and Prop. \ref{smalllo}. 
This concludes the proof.
\end{proof}

\begin{Remark}
We believe that Theorem~\ref{mainth-1} could alternatively be proved from the 
analogous description of the quotient $\Gamma\backslash\mathscr{B}_P$, where 
$\Gamma = K^{*}\mathrm{GL}_n(\A)$, given in \cite{BravoLoisel}.
Indeed, since $\Gamma$ and $N$ are commensurable whenever $k$ is
finite, it would suffice to show that the cuspidal decomposition results stated in 
Theorem~\ref{mainth-1} are, in general, inherited by the whole 
commensurability class of any group $G \subseteq \mathrm{GL}_n(K)$ satisfying them.
This is presumably accessible by purely topological and combinatorial arguments. However, 
it seems to us that the proof presented here yields a simpler and more self-contained 
exposition in our context.   
\end{Remark}

\begin{proof}[Proof of Th. \ref{mainth3}]
The existence of the ramified cover follows from the fact that
$\Gamma K^*$ is a normal subgroup of $N$. The fact that cuspidal sectors
are unramified follows from Prop. \ref{HNap35} above. Now assume $k$ is 
finite, and therefore $\mathrm{Pic}(\A)$ is finite.
In particular, the 
annihilator of $n$ is isomorphic to $\mathrm{Pic}(\A)/\mathrm{Pic}(\A)^n$.
Furthermore, $\mathrm{Pic}(\A)\cong\mathrm{Pic}(X)/\langle[P]_X\rangle$, 
where $[P]_X\in \mathrm{Pic}(X)$ is the class of the divisor $P$.
We conclude that  $\mathrm{Pic}(\A)/\mathrm{Pic}(\A)^n$ is isomorphic to 
$$T=\mathrm{Pic}(X)/\big(\langle[P]_X\rangle+\mathrm{Pic}(X)^n\big).$$ 
In particular, we have $|T|=|\mathrm{Pic}(X)/\mathrm{Pic}(X)^n|/d$. 
By Lemma \ref{lemma N/G}, this is also the order of $N/\Gamma$.
The abelian group $N/\Gamma$ acts transitively on the pre-image 
$W=\{w_1,\cdots,w_r\}\subseteq\mathscr{S}_P(D)$ of a given vertex $v\in\mathscr{C}_P(D)$.
It suffices, therefore, to find the stabilizer $S$ of a vertex 
$w_i$ which coincides with the kernel of the action on $W$, as
the group is abelian. 
The number of pre-images is $r=|T|/|S|$.
We need to prove that $S$ is trivial when so is $\mathrm{Pic}_0(X)$.

Let $\DD$ be the maximal order 
corresponding to any pre-image, in $\mathscr{B}_P(D)$, of the 
vertex $v$. A class $\bar{g}\in N/\Gamma$ is in $S$ if 
and only if $g\DD g^{-1}=\gamma\DD\gamma^{-1}$ for some
element $\gamma\in\Gamma$. This is equivalent to 
$h\DD h^{-1}=\DD$ for $h=\gamma^{-1}g$. Note that 
$\bar h=\bar g$. We conclude that $S$ is the image in 
$N/\Gamma$ of the normalizer of $\DD$. If $\DD$ is the endomorphism 
ring of a lattice $\Lambda$, then $h\Lambda=\LL^{B'}\Lambda$ for
some divisor $B'$. Note that the isomorphism in Lemma \ref{lemma N/G}
sends the class $\bar{g}\in N/\Gamma$ of any $g\in N$, 
to the class $[J]\in\mathrm{Pic}(\A)$ of the ideal satisfying
$g\Lambda(U)=J\Lambda(U)$.
In particular, $\bar h$ corresponds to the class of $B'$
in the annihilator of $n$ in $\mathrm{Pic}(\A)$. Taking determinants in
the relation $h\Lambda=\LL^{B'}\Lambda$ we see that $nB'$ is principal.
In particular $[B']\in\mathrm{Pic}_0(X)\subseteq
\mathrm{Pic}(X)$. The result follows.
\end{proof}

\begin{proof}[Proof of Th. \ref{mainth5}]
For the first statement, we note that the number $m_{\BB,P,\beta}$
of cuspidal sector chambers in the QSA $\widehat{\Sigma}_{\BB,P,\beta}$ 
is $m_{\BB,P,\beta}=\frac{n!}{|H|}$. Since $H$ is the stabilizer of the
corresponding class $\BBB\in\nabla_{\mathrm{Pic}(\A)}\cong
\mathrm{Pic}(\A)^{n-1}$, the $S_n$-orbit $[\BB]$ containing $\BBB$ has 
also $\frac{n!}{|H|}$ elements. If we write
$\Pi=S_n\backslash\nabla_{\mathrm{Pic}(R)}$, we conclude that the total 
number of cuspidal sectors is 
\begin{equation}\label{countthem}m_{\mathrm{Total}}=
    \sum_{[\BB]\in \Pi}m_{\BB,P,\beta}
    =\sum_{[\BB]\in \Pi}\frac{n!}{|H|}
    =\sum_{[\BB]\in \Pi}\left|[\BB]\right|=\left|\mathrm{Pic}(R)^{n-1}\right|.
\end{equation}
For the second statement, recall that the
number of connected components equals the number of conjugacy classes of
maximal $\A$-orders. It follows from Lemma \ref{lemma spinor genus} and 
Lemma \ref{afinsgescl} that the latter equal $u=|\mathrm{Pic}(\A)/
\mathrm{Pic}(\A)^n|$, so the result follows from Eq. (\ref{fpica}).
For the last statement we claim that, if $\BB=(B_1,\dots,B_n)$ and
$\BB'=(B'_1,\dots,B'_n)$, then the affine orders $\DD_{\BB,\beta}(U)$
and $\DD_{\BB',\beta'}(U)$ are in the same spinor genus if and only if
the divisor $\sum_{i=1}^n(B'_i-B_i)$ has trivial image in 
$\mathrm{Pic}(\A)/\mathrm{Pic}(\A)^n$. This condition divides 
$\nabla_{\mathrm{Pic}(R)}$ into $u$ $S_n$-invariant subsets of size
$\frac1u\left|\mathrm{Pic}(R)^{n-1}\right|$. 
If $\Pi_1,\dots,\Pi_u$ are the corresponding orbit sets, we can 
replace $\Pi$ by each $\Pi_i$ in Eq. (\ref{countthem}) to finish the proof.

It remains to prove the claim. Since $\DD_{\BB',\beta'}$ and $\DD_{\BB',\beta}$
are related by a global change of basis, and therefore they are in the same 
class, we can assume $\beta=\beta'$. Then we note that $\DD_{\BB',\beta}=
d\DD_{\BB,\beta}d^{-1}$, where $d$ is, in the basis $\beta$, a diagonal 
adelic matrix whose $i$-th diagonal component generates the divisor 
$B'_i-B_i$. The result follows.
\end{proof}

%\end{comment}

\section{Examples}\label{section examples}

The QSAs introduced in \S\ref{split apartments} are indeed critical 
in the description of the C-complex introduced in
\S \ref{section Quotient Buildings and conjugacy classes}.
In this section, we present a few of the simplest examples of 
these structures, as well as the 
corresponding cuspidal structures.
For $n\geq3$, in Examples \ref{E66} to \ref{e10}, we concentrate 
on the case $X=\mathbb{P}^1$.
We use, throughout, the following fact:

\begin{prop}
Assume $X=\mathbb{P}^1$. Any vertex $v$ corresponding to a lattice
of the form $\LL^{B_1}\times\LL^{B_2}\times\LL^{B_3}$ where the distance
between the numbers $\deg(B_1)$, $\deg(B_2)$ and $\deg(B_3)$ is at least
$\deg(P)-1$ lies inside of a good sector.
\end{prop}

\begin{proof}
    By the discussion preceding Lemma \ref{enoghcpa},
    it suffices to observe that
    $\LL^B(X)/\LL^{B-P}(X)$ has dimension $\deg(P)$ over $k$
    as soon as $\deg(B)\geq\deg(P)-1$.
    This is a consequence of Riemann-Roch Theorem.
    Now the result follows as in the proof of
    Lemma \ref{enoghcpa}. 
\end{proof}

\begin{ex}\label{e1}
    The trivial class in $\nabla_{\mathrm{Pic}(\A)}$ is invariant  
    under the whole of $S_n$. The corresponding action on
    $\Sigma_{\vec 0,P,\beta}$ leaves invariant
    the vertex corresponding to $\DD_{\vec 0,\beta}$.
    If $n=3$, the corresponding QSA looks like 
    the one depicted in Fig. \ref{f2}.
\begin{figure}[h!]
\begin{center}
\begin{tikzpicture}[x=1cm, y=0.866cm] % 0.866 is sin(60), fixing the aspect ratio for equilateral triangles
\fill[gray!30] (4,2) -- (8,2) -- (9,4) -- (5,4) -- cycle;
\fill[gray!60] (5.5,3) -- (8.5,3) -- (9,4) -- (6,4) -- cycle;
\node  at (5,4) {$\bullet$};
\node  at (5,4.1) {${}^{{}^{(2,2)}}$};
\node  at (6,4) {$\bullet$};
\node  at (6,4.1) {${}^{{}^{(2,3)}}$};
\node  at (4.5,3) {$\bullet$};
\node  at (4.65,3.1) {${}^{{}^{(1,1)}}$};
\node  at (5.5,3) {$\bullet$};
\node  at (5.65,3.1) {${}^{{}^{(1,2)}}$};
\node  at (4,2) {$\bullet$};
\node  at (4.15,2.1) {${}^{{}^{(0,0)}}$};
\node  at (5,2) {$\bullet$};
\node  at (5.15,2.1) {${}^{{}^{(0,1)}}$};
\node (A1) at (4.5, 1.66) {};
\node (B1) at (4.5, 2.33) {};
\draw[<->] (A1) edge (B1);
\node (A2) at (3.5, 1.66) {};
\node (B2) at (3.5, 2.33) {};
\draw[<->] (A2) edge (B2);
\node (A3) at (4, 2.66) {};
\node (B3) at (4.5, 2.33) {};
\draw[<->] (A3) edge (B3);
\node (A4) at (4, 2.66) {};
\node (B4) at (3.5, 2.33) {};
\draw[<->] (A4) edge (B4);
\node (A5) at (4, 1.33) {};
\node (B5) at (4.5, 1.66) {};
\draw[<->] (A5) edge (B5);
\node (A6) at (4, 1.33) {};
\node (B6) at (3.5, 1.66) {};
\draw[<->] (A6) edge (B6);
  % Loop for the grid rows and columns
  \foreach \row in {1,...,3} {
    \foreach \col in {1,...,6} {
      \pgfmathsetmacro{\xshift}{\col + 0.5*\row}
      \coordinate (A) at (\xshift, \row);
      \coordinate (B) at (\xshift + 1, \row);
      \coordinate (C) at (\xshift + 0.5, \row + 1);
      \coordinate (D) at (\xshift + 1.5, \row + 1);
      \draw[thick] (A) -- (B) -- (C) -- cycle;
      \draw[thick] (B) -- (C) -- (D) -- cycle;
    }
  }

\end{tikzpicture}
\end{center}
\vspace{-0.5cm}
\caption{Fundamental domain (gray) for the action of $S_3$ on 
$\Sigma_{\vec0,P,\beta}$ and the corresponding cusp inside of it
(darker gray).}
\label{f2}
\end{figure}
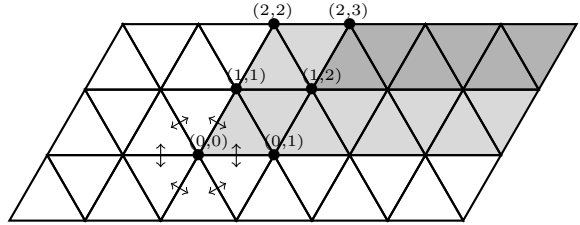
    In the case $X=\mathbb{P}^1$, the vertex labeled $(m_1,m_2)$ 
    corresponding to the order $\DD_{(0,m_1,m_2)P,\beta}$ is in 
    the cusp when $0<m_1<m_2$. In the general case the cusp could 
    be smaller.
\end{ex}

\begin{ex}
    If $n\geq 3$, then only the trivial class in 
    $\nabla_{\mathrm{Pic}(\A)}$ is invariant under the whole of $S_n$.
    In fact, if $\vec b\in\mathrm{Pic}(\A)^n$ is a representative
    of an invariant class $\tilde{b}\in\nabla_{\mathrm{Pic}(\A)}$,
    then, for any element $\sigma\in S_n$, we have
    $\sigma\cdot\vec b=\vec b+\vec \delta b_\sigma$, for some element 
    $b_\sigma\in\mathrm{Pic}(\A)$, and it is easy to see that 
    $\sigma\mapsto b_\sigma$ is a group homomorphism of
    $S_n$ into an abelian group, and must therefore be trivial on 
    $A_n$. This means that $\sigma\cdot\vec b=\vec b$ when 
    $\sigma\in A_n$, but when $n\geq 3$ then $A_n$ is transitive
    on the coordinates. It follows that $\vec b=c\vec \delta$,
    for some $c\in\mathrm{Pic}(\A)$.
\end{ex}

\begin{ex}\label{example n=2}
When $n=2$, the split apartment $\Sigma=\Sigma_{\BB,P,\beta}$ is the union 
of two rays in a tree whose intersection consists of a single vertex. 
This complex is also called a maximal path in works such as 
\cite{ArenasBranches, ArenasBravoExt, ArenasBravoCar2}.
In this case one has $H \subseteq S_2$. Observe that $H=\lbrace\mathrm{id}\rbrace$ if and only if the class
$(\vec B)\in\mathrm{Pic}(\A)^2$ has the form $(\vec B)=(b_1,b_2)$
where $2b_2\neq2b_1$. In that situation, the quotient $H \backslash \Sigma$ is naturally homeomorphic to $\Sigma$.
Moreover, if $H = S_2$, i.e., if $2b_2=2b_1$, then the quotient space 
$H \backslash \Sigma$ is homeomorphic either to a ray or to the union 
of a ray with a half edge, attached at the endpoint of the ray.
More precisely, if the image of $\vec B$ in $\mathrm{Pic}(X)^2$
is $(\tilde b_1,\tilde b_2)$, then $2(\tilde{b_2}-\tilde{b_1})=
N\bar P$, for some integer $N$, where $\bar P$ is the image of
the prime divisor $P$. Here $H \backslash \Sigma$ looks like 
Fig. \ref{figure 3b}(A) or Fig. \ref{figure 3b}(B) according to 
whether $N$ is even or not. When $X=\mathbb{P}^1$ and $\deg(P)=2$
the cusp starts at the vertex denoted $v$. The initial vertex in 
the trivial component is not part of the cusp, since it has an 
additional half edge. See \cite[\S 4]{ARENASCARMONA2014258} for 
further details. For $\deg(P)>2$, or for curves whose Picard group
has torsion,  the cusps in (B) are necessarily smaller, since 
we have edges connecting multiple QSAs.

\begin{figure}
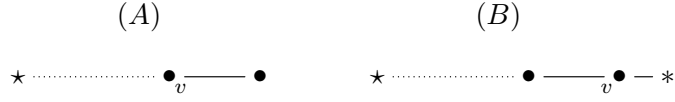

\[ 
\xygraph{
!{<0cm,0cm>;<.8cm,0cm>:<0cm,.8cm>::}
 !{(-2,3) }*+{(A)}="na"
 !{(0,2) }*+{\bullet}="x1" 
 !{(-1.5,2) }*+{\bullet}="x2" !{(-1.3,1.8) }*+{{}_{v}}
 !{(-4.,2) }*+{\star}="x3"
 !{(-3.5,2) }*+{}="x4"
"x1"-"x2" "x2"-@{.}"x3" 
 } 
\quad \quad \quad
\xygraph{
!{<0cm,0cm>;<.8cm,0cm>:<0cm,.8cm>::}
 !{(-2,3) }*+{(B)}="na"
 !{(0,2) }*+{\bullet}="x1" !{(-.2,1.8) }*+{{}_{v}}
 !{(0.8,2) }*+{*}="x0" 
 !{(-1.5,2) }*+{\bullet}="x2" 
 !{(-4.,2) }*+{\star}="x3"
"x1"-"x2" "x2"-@{.}"x3" 
"x1"-"x0" } 
\]
\caption{The quotient space $\hat{\Sigma}_{\BB,P,\beta} \cong 
H\backslash\Sigma_{\BB,P,\beta}$ in Ex. \ref{example n=2}, when 
$N$ is even in Fig. (A) and odd in Fig. (B).
In the latter case, the segment connecting $\bullet$ with $*$ 
corresponds to a half edge. Stars denote visual limits.} \label{figure 3b}
\end{figure}

\end{ex}

The next three examples completely describe the cusp structure of the full 
C-complex whenever $X=\mathbb{P}^1$ and $\deg(P)=3$, with its three 
connected components.

\begin{ex}\label{E66}
    Assume now $X=\mathbb{P}^1$ and assume $P$ is a point 
    of degree $3$, so that $\mathrm{Pic}(\A)$ is a cyclic 
    group with three elements. Set $n=3$.
    Let $Q$ be a place of degree $1$ and consider the split 
    apartment
    $\Sigma_{\BB,P,\beta}$ where $\BB=(0,1,1)Q$, so that 
    $H=H_{(\!(\BB)\!)}$ is generated by the transposition $(23)$.
    Then a fundamental domain, as the quotient split apartment,
    looks like the one depicted in Fig. \ref{f3}. 
The vertices labeled as $(3r+1,3s+1)$ correspond to the order 
$\DD_{\BB+(0,r,s)P,\beta}$. In the fundamental domain $\mathbf{\Psi}$ we
have $r\geq s$.  The corresponding quotient complex can have
additional simplices connecting the vertices in the 
border of the half plane, or in the black line, but these are all
vertices in the corresponding C-complex $\mathscr{C}_P(D)$.
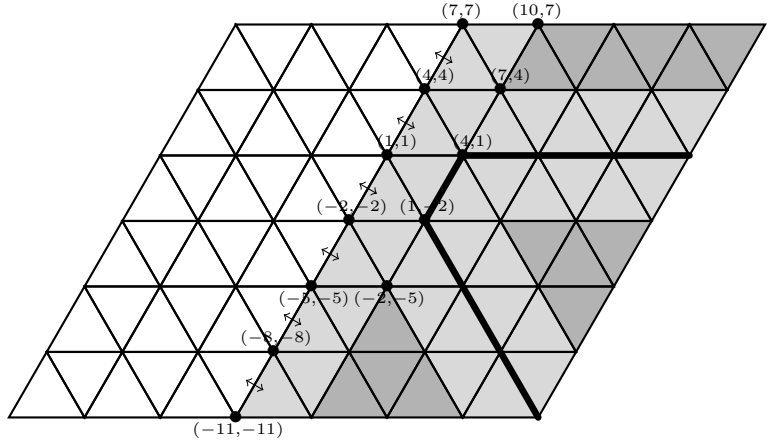
\begin{figure}[h!]
\begin{center}
\begin{tikzpicture}[x=1cm, y=0.866cm] % 0.866 = sin(60)

% --- Rellenos (van primero para que la malla quede encima) ---
% Dominio fundamental Psi (prolongado hacia abajo): gris claro
\fill[gray!30] (2,-2) -- (6,-2) -- (9,4) -- (5,4) -- cycle;
% Regiones delimitadas en rojo: gris oscuro
\fill[gray!60] (5.5,3) -- (8.5,3) -- (9,4) -- (6,4) -- cycle;
\fill[gray!60] (5.5,1) -- (7.5,1) -- (6.5,-1) -- cycle;
\fill[gray!60] (4,0) -- (3,-2) -- (5,-2) -- cycle;

% --- Malla triangular ---
\foreach \row in {-2,...,3} {
  \foreach \col in {0,...,6} {
    \pgfmathsetmacro{\xshift}{\col + 0.5*\row}
    \coordinate (A) at (\xshift, \row);
    \coordinate (B) at (\xshift + 1, \row);
    \coordinate (C) at (\xshift + 0.5, \row + 1);
    \coordinate (D) at (\xshift + 1.5, \row + 1);
    \draw[thick] (A) -- (B) -- (C) -- cycle;
    \draw[thick] (B) -- (C) -- (D) -- cycle;
  }
}

% --- Camino destacado en negro ---
\draw[line width=2.5pt, line cap=round, line join=round]
  (6,-2) -- (4.5,1) -- (5,2) -- (8,2);

% --- Puntos y etiquetas ---
\node  at (5,4) {$\bullet$};
\node  at (5,4.1) {${}^{{}^{(7,7)}}$};
\node  at (6,4) {$\bullet$};
\node  at (6,4.1) {${}^{{}^{(10,7)}}$};
\node  at (4.5,3) {$\bullet$};
\node  at (4.65,3.1) {${}^{{}^{(4,4)}}$};
\node  at (5.5,3) {$\bullet$};
\node  at (5.65,3.1) {${}^{{}^{(7,4)}}$};
\node  at (4,2) {$\bullet$};
\node  at (4.15,2.1) {${}^{{}^{(1,1)}}$};
\node  at (5,2) {$\bullet$};
\node  at (5.15,2.1) {${}^{{}^{(4,1)}}$};
\node  at (3.5,1) {$\bullet$};
\node  at (3.55,1.1) {${}^{{}^{(-2,-2)}}$};
\node  at (4.5,1) {$\bullet$};
\node  at (4.55,1.1) {${}^{{}^{(1,-2)}}$};
\node  at (3,0) {$\bullet$};
\node  at (3.05,-0.3) {${}^{{}^{(-5,-5)}}$};
\node  at (4,0) {$\bullet$};
\node  at (4.05,-0.3) {${}^{{}^{(-2,-5)}}$};
\node  at (2.5,-1) {$\bullet$};
\node  at (2.55,-0.9) {${}^{{}^{(-8,-8)}}$};
\node  at (2,-2) {$\bullet$};
\node  at (2.05,-2.3) {${}^{{}^{(-11,-11)}}$};

% --- Flechas ---
\node (A1) at (3, 0.66) {};
\node (B1) at (3.5, 0.33) {};
\draw[<->] (A1) edge (B1);
\node (A2) at (3.5, 1.66) {};
\node (B2) at (4, 1.33) {};
\draw[<->] (A2) edge (B2);
\node (A3) at (4, 2.66) {};
\node (B3) at (4.5, 2.33) {};
\draw[<->] (A3) edge (B3);
\node (A4) at (4.5, 3.66) {};
\node (B4) at (5, 3.33) {};
\draw[<->] (A4) edge (B4);
\node (A5) at (2.5, -0.34) {};
\node (B5) at (3, -0.67) {};
\draw[<->] (A5) edge (B5);
\node (A6) at (2, -1.34) {};
\node (B6) at (2.5, -1.67) {};
\draw[<->] (A6) edge (B6);

\end{tikzpicture}
\end{center}
\vspace{-0.5cm}
\caption{Fundamental domain (gray) by the action of 
$H=\left\langle(2\,3)\right\rangle$ on $\ap_{\BB,P,\beta}$ with
$\BB=(0,1,1)Q$, when $X=\mathbb{P}^1$, $\deg P = 3$ and $\deg Q = 1$.
The cusps are in darker gray. The vertices in the black path 
can potentially have additional neighbors.  }
\label{f3}
\end{figure}
\end{ex}

\begin{ex}
    If we replace $(0,1,1)$ by $(0,2,0)$ in the preceding example, 
    we obtain the same quotient complex, since both vector divisors
    $(0,1,1)Q$ and $(0,2,0)Q$ correspond to the same orbit in 
    $\nabla_{\mathrm{Pic}(\A)}$. There is a correspondence between both
    complexes that preserves the isomorphism class of the orders associated
    to their vertices. On the other hand, the vector 
    $\BB=(0,1,0)Q$, corresponds to a different orbit, so the
    corresponding QSA, although having a picture very similar 
    to the preceding one, has vertices defined by non-isomorphic
    maximal orders.
\end{ex}

\begin{ex}\label{e7}
    Now set $\BB=(0,1,2)Q$. Then the stabilizer $H$ is generated by 
    the $3$-cycle $(1\,2 \,3)$. This element induces a rotation by 
    $120^\circ$ on the apartment.
    The rotation is induced by the matrix $g=\textnormal{\scriptsize$\left(\begin{array}{ccc}0&1&0\\0&0&1\\f&0&0\end{array}\right)$\normalsize}$,
    where $\mathrm{div}(f)=P-3Q$. 
    This element permutes cyclically the three vertices of a unique
    simplex, namely the orders $\DD_{(0,1,2)Q,\beta}$, 
    $\DD_{(2,0,1)Q,\beta}$ and $\DD_{(1,2,0)Q,\beta}$, corresponding to 
    the points labeled $(1,2)$, $(-2,-1)$ and $(1,-1)$. The center of 
    this simplex is marked with a symbol $\clubsuit$ in Fig. \ref{f4}.
    The C-complex $\mathscr{C}_P(D_0)$, defined by the order 
    $D_0=\mathbb{M}_3(R)$, can be obtained by 
    connecting the border of the fundamental domain in Fig. \ref{f2}
    to the black line in Fig. \ref{f4}.

\begin{comment}
\begin{figure}[h!]
\begin{center}
\begin{tikzpicture}[x=1cm, y=0.866cm] % 0.866 is sin(60), fixing the aspect ratio for equilateral triangles
\fill[gray!30] (2,4) -- (5.5,1.66) -- (9,4) -- cycle;
\node  at (5.5,1.66) {$\clubsuit$};
\node  at (5,2) {$\bullet$};
\node  at (6,2) {$\bullet$};
\node  at (5.5,1) {$\bullet$};
\node  at (5,2.2) {${}^{{}^{(1,2)}}$};
\node  at (6,2.2) {${}^{{}^{(1,-1)}}$};
\node  at (5.5,0.6) {${}^{{}^{(-2,-1)}}$};
\node  at (3.5,3) {$\bullet$};
\node  at (7.5,3) {$\bullet$};
\node  at (3.5,2.6) {${}^{{}^{(4,8)}}$};
\node  at (7.5,2.6) {${}^{{}^{(4,-4)}}$};
  % Loop for the grid rows and columns
  \foreach \row in {0,...,3} {
    \foreach \col in {0,...,6} {
      
      % Shift every second row to the right to align triangles properly
      \pgfmathsetmacro{\xshift}{\col + 0.5*\row}
      
      % Define the base vertex of the upward triangle
      \coordinate (A) at (\xshift, \row);
      \coordinate (B) at (\xshift + 1, \row);
      \coordinate (C) at (\xshift + 0.5, \row + 1);
      
      % Draw the upward triangle
      \draw[thick, black!100!black] (A) -- (B) -- (C) -- cycle;
      
      % Draw a downward triangle to fill the gaps (only if not on the last row)
      \ifnum\row<5
        \coordinate (D) at (\xshift + 1.5, \row + 1);
        \draw[thick, black!100!black] (B) -- (C) -- (D) -- cycle;
      \fi
    }
  }

\end{tikzpicture}
\end{center}

\caption{Fundamental domain by the action of $H=\left\langle(1\,2\,3)\right\rangle$ 
on $\Sigma_{(0,1,2)Q,\beta}$.}
\label{f4}
\end{figure}
\end{comment}

\begin{figure}[h!]
\begin{center}
\begin{tikzpicture}[x=1cm, y=0.866cm] % 0.866 = sin(60)

% --- Rellenos (van primero para que la malla quede encima) ---
% Región delimitada en naranjo: gris claro
\fill[gray!30] (0.8333,1.6667) -- (5.5,1.6667) -- (7.6667,6) -- (3,6) -- cycle;
% Región delimitada en rojo: gris oscuro (parte en (4,8))
\fill[gray!60] (1.5,3) -- (3.5,3) -- (2.5,5) -- cycle;
% Región delimitada en verde: gris oscuro
\fill[gray!60] (5.5,3) -- (7,6) -- (4,6) -- cycle;

% --- Malla triangular ---
\foreach \row in {0,...,5} {
  \foreach \col in {0,...,6} {
    \pgfmathsetmacro{\xshift}{\col + 0.5*\row}
    \coordinate (A) at (\xshift, \row);
    \coordinate (B) at (\xshift + 1, \row);
    \coordinate (C) at (\xshift + 0.5, \row + 1);
    \coordinate (D) at (\xshift + 1.5, \row + 1);
    \draw[thick] (A) -- (B) -- (C) -- cycle;
    \draw[thick] (B) -- (C) -- (D) -- cycle;
  }
}

% --- Camino grueso por (1,5), (1,2) y (4,5) ---
\draw[line width=2.5pt, line cap=round, line join=round]
  (1,2) -- (5,2) -- (3,6);

% --- Puntos ---
\node at (5.5,1.6667) {$\clubsuit$};
\node at (3.5,3) {$\bullet$};
\node at (4.5,3) {$\bullet$};
\node at (5.5,3) {$\bullet$};
\node at (4,2)   {$\bullet$};
\node at (5,2)   {$\bullet$};
\node at (6,2)   {$\bullet$};
\node at (5.5,1) {$\bullet$};

% --- Etiquetas ---
\node at (3.5,3.3)   {${}^{{}^{(4,8)}}$};
\node at (4.5,3.3)   {${}^{{}^{(4,5)}}$};
\node at (5.5,3.3)   {${}^{{}^{(4,2)}}$};
\node at (4,2.2)     {${}^{{}^{(1,5)}}$};
\node at (5,2.2)     {${}^{{}^{(1,2)}}$};
\node at (6,2.2)     {${}^{{}^{(1,-1)}}$};
\node at (5.5,0.6)   {${}^{{}^{(-2,-1)}}$};

\end{tikzpicture}
\end{center}
\vspace{-0.5cm}
\caption{Fundamental domain by the action of $H=\left\langle(1\,2\,3)\right\rangle$ 
on $\Sigma_{(0,1,2)Q,\beta}$.}
\label{f4}
\end{figure}
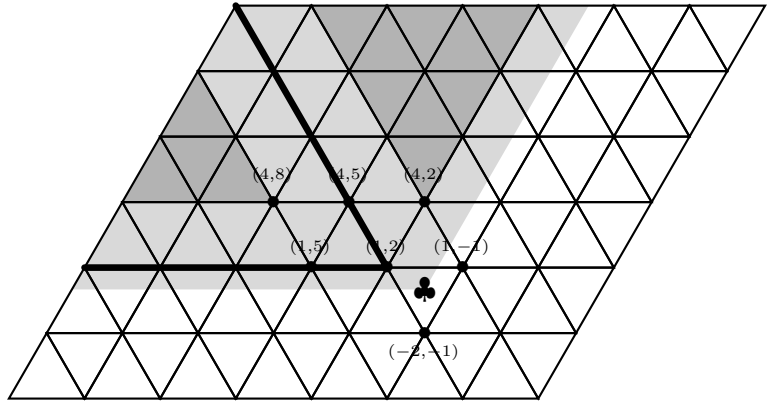
\end{ex}

The next example discusses the cusp structure of the C-complex when
$X=\mathbb{P}^1$ and $\deg(P)=2$.

\begin{ex}\label{ex9}
The $1$-skeleton of the C-complex when $X$ is the projective line 
and $P$ is a point of degree $2$ is illustrated on page 49 
of \cite{Christ}. The vertex $X_{n,m}$ in the reference corresponds
to the maximal order $\DD_{[[m-n,-n]]Q}\cong\DD_{[[n,m]]Q}$ in
our notation. In this case, the group $\nabla_{\mathrm{Pic}(\A)}$
has two orbits under the symmetric group, namely the identity, 
and the set $\{[[a,0]],[[0,a]],[[a,a]]\}$, where $a$ is the non-trivial
class. The QSA corresponding to non-trivial class is computed in a way 
that is entirely analogous to Example \ref{E66}. In Fig. 
\ref{Figlast}(A), we illustrate a fundamental region for this 
QSA in the apartment $\Sigma_{\BB,P,\epsilon}$ with $\BB=(0,Q,Q)$. 
The vertex labeled with an odd number $n$ corresponds to the maximal 
order $\DD_{(0,Q,Q)+\frac{n-1}2(0,P,P),\epsilon}\cong\DD_{[[n,n]]Q}$.
Since $\DD_{[[n,n]]Q}\cong\DD_{[[0,-n]]Q}$, the vertices labeled
with negative numbers can be identified with those corresponding to 
the orders $\DD_{[[0,1]]Q},\DD_{[[0,3]]Q},\dots$. This allows us to
identify them with the corresponding vertex in the image in the 
reference, where the vertex corresponding to $\DD_{[[i,j]]Q}$
is denoted $X_{i,j}$. We conclude that the subdivision of the 
quotient graph looks as shown in Fig. \ref{Figlast}(B).

\begin{figure}
    \centering
    (A)\begin{tikzpicture}[scale=0.8]

    % --- Rellenos ---
    % Dos filas inferiores: gris claro
    \fill[gray!30] (0,0) -- (5,0) -- (6,1.732) -- (1,1.732) -- cycle;
    % Regiones delimitadas en rojo: gris oscuro
    \fill[gray!60] (0,0) -- (2,0) -- (2.5,0.866) -- (0.5,0.866) -- cycle;
    \fill[gray!60] (3,0) -- (4,0) -- (3.5,0.866) -- cycle;
    \fill[gray!60] (5,0) -- (5.5,0.866) -- (4.5,0.866) -- cycle;

    % --- Malla triangular ---
    \foreach \row in {0,...,3} {
        \foreach \col in {0,...,4} {
            \pgfmathsetmacro{\x}{\col + (\row)*0.5}
            \pgfmathsetmacro{\y}{\row * sqrt(3)/2}
            \pgfmathsetmacro{\xnext}{\x + 1}
            \pgfmathsetmacro{\xup}{\col + (\row+1)*0.5}
            \pgfmathsetmacro{\yup}{(\row + 1) * sqrt(3)/2}
            \pgfmathsetmacro{\xupnext}{\xup + 1}
            \draw[thick] (\x,\y) -- (\xnext,\y) -- (\xup,\yup) -- cycle;
            \draw[thick] (\xnext,\y) -- (\xup,\yup) -- (\xupnext,\yup) -- cycle;
        }
    }

    % --- Puntos y etiquetas ---
    \node at (2,1.73) {$\bullet$};
    \node at (1.7,1.83) {\tiny ${}^{-3}$};
    \node at (3,1.73) {\scriptsize$\spadesuit$\normalsize};
    \node at (2.7,1.83) {\tiny ${}^{-1}$};
    \node at (4,1.73) {\scriptsize$\clubsuit$\normalsize};
    \node at (3.75,1.83) {\tiny ${}^{1}$};
    \node at (5,1.73) {$\bullet$};
    \node at (4.75,1.83) {\tiny ${}^{3}$};
\end{tikzpicture}
\qquad\qquad
    (B)\begin{tikzpicture}[scale=0.8]

    % --- Rellenos ---
    % Fondo gris claro
    \fill[gray!30] (0,0) rectangle (6,6);
    % Regiones en blanco
    \fill[white] (0,2) rectangle (3,3);
    \fill[white] (3,3) rectangle (4,6);
    \fill[white] (3,2) -- (3,3) -- (4,3) -- cycle;
    % Regiones delimitadas en rojo: gris oscuro
    \fill[gray!60] (0,4) rectangle (2,6);
    \fill[gray!60] (5,3) rectangle (6,6);
    \fill[gray!60] (4,0) rectangle (6,2);
    \fill[gray!60] (0,0) rectangle (3,1);

    % --- Cuadrícula ---
    \draw[thick] (0,0) grid (6,6);

    % --- Diagonales "/" ---
    \foreach \x in {3,4,5} {\foreach \y in {0,1,2} {\draw[thick] (\x,\y) -- ++(1,1);}}
    \foreach \x in {0,1,2} {\foreach \y in {3,4,5} {\draw[thick] (\x,\y) -- ++(1,1);}}

    % --- Diagonales "\" ---
    \foreach \x in {0,1,2} {\foreach \y in {0,1,2} {\draw[thick] (\x,\y+1) -- ++(1,-1);}}
    \foreach \x in {3,4,5} {\foreach \y in {3,4,5} {\draw[thick] (\x,\y+1) -- ++(1,-1);}}

    % --- Símbolos ---
    \node at (4,3) {$\spadesuit$};
    \node at (3,2) {$\clubsuit$};
    \node at (3,3) {$\blacksquare$};
\end{tikzpicture}
    \caption{The two QSAs for the quotient complex corresponding to 
    the projective line at a place of degree 2. The square corresponds
    to the maximal order $\DD_{0}$, while $\clubsuit$ and $\spadesuit$
     correspond to the orders $\DD_{[[Q,Q]]}$ and $\DD_{[[Q,0]]}$, 
     respectively, for a degree-1 place $Q$.}
    \label{Figlast}
\end{figure}
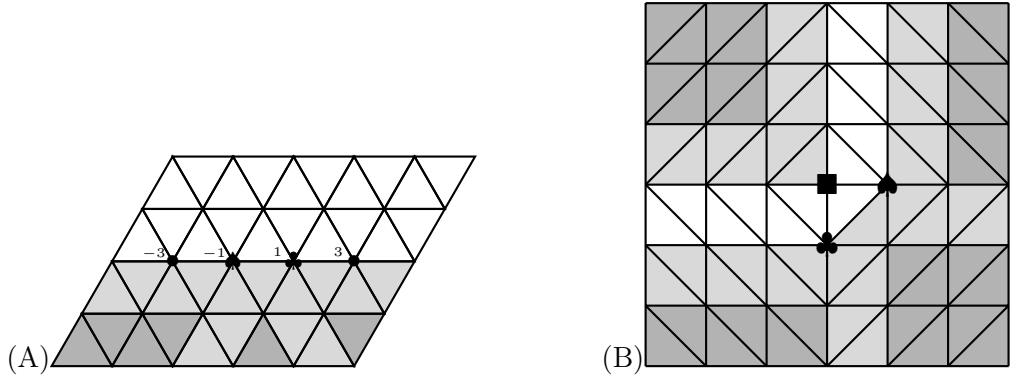
\end{ex}

\begin{ex}
    More generally, for $X=\mathbb{P}^1$, and for a point $P$ of degree $d$,
    there is a unique orbit of size $1$ and $d-1$ orbits of size $3$. An orbit
    of size $2$ (corresponding to a cone) appears only when $d=3k$
    and it is the orbit $\{(\!(\vec B)\!),(\!(\vec B')\!)\}$, where
    $\vec B=(0,k,2k)Q$ and $\vec B'=(0,2k,k)Q$. 
    The number of orbits of size
    $6$ is $\frac{d^2-3d}6$ if $3$ divides $d$ and $\frac{d^2-3d+2}6$ otherwise.
\end{ex}

\begin{ex}\label{e10}
The quotients studied in this work provide information on orders that can be
described in terms of the maximal orders containing them. As an example, 
the chambers 
of the complex correspond to conjugacy classes of a 
sheaf-theoretical analog of 
Iwahori orders at $P$. A local Iwahori order is defined 
as the intersection of the
maximal orders corresponding to the vertices in a simplex, 
and a global Iwahori order
at $P$ is an order $\mathfrak{W}$ whose completion $\mathfrak{W}_P$ is a local
Iwahori order, while any other completion $\mathfrak{W}_Q$ is maximal. 
These are examples of hereditary orders, see \cite[\S 39]{Reiner}.
Every white triangle in Fig. \ref{Figlast}(B) corresponds
to a different conjugacy
class of global Iwahori orders at $P$ that are non-split, 
in the sense that they
cannot be written in the form 
$$\mathfrak{W}=\left(\begin{array}{ccc}
\oo_X&\LL^{B_{12}}&\LL^{B_{13}}\\
\LL^{B_{21}}&\oo_X&\LL^{B_{23}}\\
\LL^{B_{31}}&\LL^{B_{32}}&\oo_X
\end{array}\right),$$
for any divisors $B_{ij}$. The figure proves, therefore, the existence 
of infinitely many conjugacy classes of such orders in that case.
\end{ex}

\footnotesize{
\section*{Acknowledgements}
The second author was partially supported by ANID through 
Fondecyt Initiation Grant No. 11260422.
The authors would like to thank Marion Jeannin for 
useful remarks on vector bundles.

\bibliographystyle{abbrv}
\bibliography{sample}
}

\end{document}